\documentclass[10pt]{article}
\usepackage[bottom=2cm,top=2cm,left=2.5cm,right=2.5cm]{geometry}
\usepackage{amsmath} 
\usepackage{amssymb} 
\usepackage{lmodern} 
\usepackage[table]{xcolor} 
\usepackage{mathtools} 
\usepackage{xspace} 
\usepackage{fancybox} 
\usepackage{graphicx} 
\usepackage{thmtools} 
\usepackage{mathrsfs}
\usepackage{booktabs}
\usepackage[english]{babel}
\usepackage{mathrsfs}
\usepackage{booktabs}
\usepackage{hyperref}
\usepackage{amsmath} 
\usepackage{amssymb} 
\usepackage{lmodern} 
\usepackage{mathtools} 
\usepackage{xspace} 
\usepackage{mathrsfs}
\usepackage{booktabs}

\usepackage{hyperref}
\hypersetup{
	colorlinks   = true,
	citecolor    = blue, urlcolor = black,
}
\usepackage{fancyhdr}
\usepackage{xcolor}
\usepackage{amsfonts}
\usepackage[Bjornstrup]{fncychap}

\usepackage{multicol}

\newcommand{\F}{\mathcal{F}}
\newcommand{\Hilbert}{\mathbb{H}}
\newcommand{\Proba}{\mathbb{P}}
\newcommand{\Borel}{\mathcal{P}}
\newcommand{\Y}{\mathbb{Y}}
\newcommand{\Kilbert}{\mathbb{K}}
\newcommand{\Ex}{\mathbb{E}}
\newcommand{\Real}{\mathbb{R}}

\newcommand{\Linear}{\mathcal{L}}
\title{\textbf{On weighted pseudo almost automorphic mild solutions for some mean field stochastic evolution equations.  }}

\author{\fontsize{10pt}{20pt}\selectfont  Moustapha Dieye $^{(1,5)}$\footnote{E-mail address: mdieye@ept.sn}, Amadou Diop $^{(2)}$\footnote{E-mail address: diop.amadou@ugb.edu.sn}, Mamadou Moustapha Mbaye $^{(4)}$\footnote{E-mail address: mamadoumoustapha3.mbaye@ucad.edu.sn} \, and  Mark  A. McKibben $^{(3)}$\footnote{E-mail address: mmckibben@wcupa.edu
}
}
\date{\today}
\begin{document}

	\newtheorem{theorem}{Theorem}[section]
	\newtheorem{corollary}[theorem]{Corollary}
	\newtheorem{lemma}[theorem]{Lemma}
	\newtheorem{example}{Example}[section]
	\newtheorem{definition}{Definition}[section]
	\newtheorem{remark}{Remark}[section]
	\renewenvironment{proof}{$\mathbf{Proof. \quad}$}{\begin{flushright}$\checkmark$\end{flushright}}

	\maketitle
	\noindent

{\fontsize{7.5pt}{10pt}\selectfont
\begin{tabular}{ll}
$^1$& École Polytechnique de Thiés, Département Tronc commun, Thiés BP A10, Sénégal \\
$^2$ &Universit\'e Gaston Berger de Saint-Louis, \,UFR SAT, D\'epartement de Math\'ematiques, B. P. 234, \,Saint-Louis S\'en\'egal \\
$^3$ & Department of Mathematics West Chester University, 25 University Avenue, West Chester, PA, 19383, U.S.A. \\
$^4$ &D\'epartement de Math\'ematiques, \,Facult\'{e} des Sciences et Technique, \,Universit\'{e} Cheikh Anta Diop, BP 5005, \,Dakar-Fann, Senegal\\
$^5$ & African Institute for Mathematical Sciences Ghana, P.O. Box LGDTD 20046, Summerhill Estates, East Legon Hills, Santoe, Acrra
\end{tabular}
}
	\abstract{ When the evolution familiy is hyperbolic and satisfies the Acquistapace-Terreni conditions, the existence and uniquenness of an almost automorphic mild solution and a weighted pseudo almost automorphic mild solution in distribution of  mean-filed nonautonomous stochastic evolution equations driven by fractional Brownian motion is proved. Examples illustrating the main results are included.				}

	\fontsize{12pt}{20pt}\selectfont

\section{Introduction}
The aim of this work is to study the existence of  mean-square almost automorphic and weighted pseudo almost automorphic mild solutions in distribution to the following class of mean field stochastic evolution equations driven by a fractional Brownian motion in a separable Hilbert space $\Hilbert$ :
\begin{equation}\label{C1}
\begin{array}{rl}
d\vartheta(t)=&A(t)\vartheta(t)\,dt + f(t, \vartheta(t), \Proba_{\vartheta(t)})\,dt + \theta(t, \vartheta(t), \Proba_{\vartheta(t)})\,dW(t)\\\\
 &+\psi(t, \Proba_{\vartheta(t)})\,dB^H(t)\quad  \mbox{for all}\quad t \in \mathbb{R},
\end{array}
\end{equation}
where $\{A(t)\}_{t \in \mathbb{R}}$ is a family of densely-defined closed linear operators satisfying the Acquistapace-Terreni conditions; $\Proba_{\vartheta(t)}$ denotes the probability measures induced by $\vartheta(t)$; $f,\theta$ and $\psi$ are stochastic processes ; $B^H=\big\{B^H_t,  t\in\mathbb{R}\big\}$ is a cylindrical fractional Brownian motion (fBm) with Hurst parameter $H\in (1/2, 1)$ with values in a separable Hilbert space $\mathbb{U}$ and  $W(t)$ is a two-sided and standard one-dimensional Brownian motion on a separable Hilbert space $\mathbb{U}_0$ independent of $B^H$.\\
During the last two decades, stochastic differential equations driven by fractional Brownian motions have been considered extensively.
The main difficulty encountered when studying  the stochastic evolution equation (\ref{C1}) is due- the fact that the fBm is neither a Markov process nor a semimartingale, excepted for $H=\frac{1}{2}$. Thus, the usual stochastic calculus cannot be applied. There are essentially two different ways to define stochastic integrals with respect to fBm.
One developed by Ciesielski, Kerkyacharian and Roynette \cite{CKR} and Z\"ahle \cite{Z} is a path-wise approach that uses the H\"older continuity properties of the sample paths. The otherintroduced by Dereusefond and \"Ust\"unel in \cite{DU}, is the stochastic calculus of variations (Malliavin calculus) for the fBm.
\par As a generalization of almost periodicity, the concept of almost automorphy was introduced by Bochner \cite{B}. For stochastic processes, the notion of distributionally almost automorphy for stochastic processes was considered in the articles \cite{FU1,FU2}. In this context several authors have studied the stochastic periodicity in distribution  and almost automorphic solutions in a distribution sense for stochastic differential equation, for instance see \cite{FU1,FU2,LUI1}. In \cite{CHENLIN}, the  authors Chen and Lin
introduced the concept of the square-mean weighted pseudo almost automorphy, which  is a generalization
of the square-mean pseudo almost automorphy, and established the well-posedness of the square-mean
weighted pseudo almost automorphic solutions  for a general class of non-autonomous stochastic evolution
equations that satisfy either global or local Lipschitz condition,  whereas Kexue and Li \cite{KLi} have established the existence and uniqueness results of almost automorphic solutions in
distribution and weighted pseudo almost automorphic solutions in distribution for
some semilinear nonautonomous stochastic partial differential equations driven by
L\'evy noise.\\
As it is very well known,  the works of Vlasov \cite{Vlasov}, Kac \cite{Kac} and McKean \cite{M1}, mean-field stochastic differential equations also called McKean-Vlasov equations arise from Boltzmann's equation in physics. Such SDEs are used to model weak interactions among particles in a multi-particle system. The current literature on mean-field stochastic differential equations is extensive. Many papers are devoted on the problems of McKean Vlasov differential equations and their application using different methods \cite{M1,M2,M3,M4,M5}.

Motivated by the aforementionned papers, this work focuses on the existence and uniqueness of almost automorphic mild solution and weighted pseudo almost automorphic mild solution in distribution of  McKean-Vlasov nonautonomous stochastic evolution equations driven by fractional Brownian motion of the abstract form Equ.\eqref{C1}.  This result generalizes the one in  Diop et al. \cite{Mbaye} and  Chen and Zhang \cite{CHEN2019}.

\par
This paper is organized as follows. In Sections 2, we briefly recall some basic facts regarding fractional Brownian motion, evolution families, almost automorphic  processes and weighted pseudo almost automorphic processes. In Section 3,  we study the existence and uniquness of mild almost automorphic mild solution for Equ.\eqref{C1}. In section 4, we investigate weighted pseudo almost automorphic mild solutions in distribution for Equ.\eqref{C1}. Finally, in Section 5, we provide  examples to illustrate the basic theory developed in this work.

\section{Preliminaries}
In this section we recall some concepts, results  and notations that will be used in the sequel.
Let $(\Y, d)$ be a separable, complete metric space and $\Borel(\Y)$ be the space of Borel probability measures on $\Y$. For $\mu_1,\mu_2\in \Borel(\Y)$, we define
 \begin{equation}
 	d_{BL}(\mu_1,\mu_2)=\displaystyle\sup_{\|g\|_{BL}\le 1} \left| \int_{\Y}^{}\,g\, d(\mu_1-\mu_2)\right|,
 \end{equation}
 where $g$ are Lipschitz continuous functions on $\Y$ with the norm
 \begin{equation*}
 \begin{array}{rl}
 \|g\|_{L}=&\sup \left\lbrace \dfrac{ |g(k)-g(l)|}{\|k-l\|}\; ;\, k,l\in \Y,\, k\ne l \right\rbrace\\\\
 \|g\|_{BL}=& \max \{\|g\|_{\infty}, \|g\|_{L} \}\;,\; \|g\|_{\infty}:=\displaystyle\sup_{k\in \Y}|g(k)|<\infty.
 \end{array}
 \end{equation*}
 It is known that $d_{BL}$ is a complete metric on $\Borel(\Y)$ which generates the weak topology \cite{MT}. Let $(\Hilbert,\|\cdot\|)$ be a real separable Hilbert spaces. We assume that $(\Omega,\F,(\F_t)_{t\ge 0}, \Proba)$ is a probability space, and $\mathcal{L}^2(\Proba,\Hilbert)$ stand for the space of all $\Hilbert$-valued random variables $\vartheta$ such that $\Ex\|\vartheta\|^2=\displaystyle\int_\Omega\|\vartheta(t)\|^2 d\Proba<\infty$. We denote by $\Proba_\vartheta=\Proba\circ \vartheta^{-1}=\mu(\vartheta)$ the distribution of all random variable $\vartheta: (\Omega,\F,\Proba) \to \Hilbert$. For any $\mu_1,\mu_2\in \Borel(\Hilbert)$,   the $2$-Wasserstein distance is defined by

 $$\mathcal{W}(\mu_1,\mu_2)=\inf\left\lbrace \left[\displaystyle\int_{\Kilbert\times\Kilbert}^{} |u-v|^2\pi(du,dv) \right]^{1/2} ,\pi\in \Borel(\Kilbert\times \Kilbert) \mbox{ with marginals } \mu_1 \mbox{ and } \mu_2\right\rbrace  .$$
 Note that if $\vartheta,\widetilde{\vartheta}\in \mathcal{L}^2(\Proba,\Kilbert)$, then
 $  \mathcal{W}(\Proba_{\vartheta},\Proba_{\widetilde{\vartheta}})\le \left(\Ex\|\vartheta-\widetilde{\vartheta}\|^2 \right)^{1/2} .$

\subsection{Almost automorphic  and weighted pseudo almost automorphic stochastic process }
In this section, we recall some known facts about almost automorphic processes. First, we give the following definitions
\begin{definition}
	A stochastic process $\vartheta : \mathbb{R}\to \mathcal{L}^2(\Proba,\Kilbert)$ is
	\begin{description}
		\item[(a)] $\mathcal{L}^2$-continuous if for any $t^\prime\in\mathbb{R}$, $\quad\displaystyle\lim\limits_{t\to t^\prime} \Ex\| \vartheta(t)-\vartheta(t^\prime)\|=0$,
		\item[(b)] $\mathcal{L}^2$-bounded if $\displaystyle\sup_{t\in \mathbb{R}}\Ex\|\vartheta(t)\|^2<\infty$.
	\end{description}
\end{definition}
We denote by $\mathcal{C}_b(\mathbb{R},\mathcal{L}^2(\Proba,\Kilbert))$ the Banach space of all $\mathcal{L}^2$-continuous and uniformly bounded stochastic processes endowed with the norm $\|\vartheta\|^2_\infty=\sup_{t\in \mathbb{R}}\left( \Ex\|\vartheta(t)\|^2_\Kilbert \right)$.
\begin{definition}
	An $\mathcal{L}^2$-continuous stochastic process $\vartheta : \mathbb{R}\to \mathcal{L}^2(\Proba,\Kilbert)$ is said to be square-mean almost automorphic, if for every sequence of real numbers $\{e^\prime_n\}\subset \mathbb{R}$, there exists a subsequence $\{e_n\}\subset \{e^\prime_n\}$ and a stochastic process $\widehat{\vartheta}: \mathbb{R}\to \mathcal{L}^2(\Proba,\Kilbert)$ such that
	\begin{equation}
	\lim\limits_{n\to \infty} \Ex\| \vartheta(t+e_n)-\widehat{\vartheta}(t)\|^2  \qquad\text{and}\qquad \lim\limits_{n \to \infty} \Ex\|\widehat{\vartheta}(t-e_n)-\vartheta(t)\|^2=0
	\end{equation}
	for each $t\in \mathbb{R}$.
\end{definition}
In the sequel, we denote by $SAA(\mathbb{R};\mathcal{L}^2(\Proba,\Kilbert))$ the collection of all square-mean almost automorphic stochastic processes $\vartheta: \mathbb{R} \to \mathcal{L}^2(\Proba,\Kilbert)$ and define
 \begin{align*}
  &SAA(\mathbb{R} \times \mathcal{L}^2(\Proba,\Kilbert)\times \Borel(\Hilbert),\mathcal{L}^2(\Proba,\Kilbert) )\\
   &= \bigg\{g(\cdot,\vartheta,\Proba_\vartheta) \in SAA(\mathbb{R}, \mathcal{L}^2(\Proba,\Kilbert)):\vartheta \in \mathcal{L}^2(\Proba,\Kilbert), \Proba_\vartheta \in \Borel(\Hilbert))\bigg\}.
 \end{align*}
\begin{definition}
	A continuous stochastic process $\vartheta : \mathbb{R}\to \mathcal{C}(\mathbb{R},\Kilbert)$ is almost automorphic in distribution if every sequence $\{e^\prime_n\}\subset \mathbb{R}$ has a subsequence $\{ e_n \}$ such that for some stochastic process $\widetilde{\vartheta}$ :
	$$
	\displaystyle\lim\limits_{n\to \infty}d_{BL}(\Proba\circ[\vartheta(t+e_n) ]^{-1}\,,\,\Proba\circ[\widetilde{\vartheta}(t)]^{-1})=0\mbox{ and }
	\displaystyle\lim\limits_{n\to \infty}d_{BL}(\Proba\circ [\widetilde{\vartheta}(t-e_n) ]^{-1}\,,\,\Proba\circ[\vartheta(t)]^{-1})=0
	$$
	hold, for each $t\in \mathbb{R}$. That is, the $\Borel(\mathcal{C}(\mathbb{R},\Kilbert))$-valued mapping $t\mapsto\Proba(\vartheta^{-1}(t))$ on $\mathbb{R}$  is almost automorphic.
\end{definition}

Next, we recall some facts of the notion of weighted pseudo almost automorphic process.
Let $\mathcal{M}$ be the set of all functions that are positive and locally integrable over $\mathbb{R}$. For given $q>0$ and $\rho\in \mathcal{M}$, define $$m(q,\rho)=\displaystyle\int_{-q}^{q} \rho(t)dt,$$ and $$\mathcal{M}_\infty=\{\rho\in \mathcal{M} \; :\; \displaystyle\lim\limits_{q\to +\infty} m(q,\rho)=+\infty \}.$$
By $SBC_0(\mathbb{R},\mathcal {L}^2(\Proba,\Kilbert),\rho)$ we denote the collection $\mathcal{L}^2$-bounded and $\mathcal{L}^2$-continuous process $\vartheta(t)$ such that $\displaystyle\lim\limits_{q\to +\infty} \dfrac{1}{m(q,\rho)}\int_{-q}^{q}\Ex\|\vartheta(t)\|^2\rho(t)dt=0$. From \cite{KLi}, it is known that $SBC_0(\mathbb{R},\rho)$ equipped with the norm $\|\vartheta\|_\infty$ is a Banach space.\\
Denote by
\begin{align*}
   &SBC_0(\mathbb{R}\times \mathcal {L}^2(\Proba,\Kilbert)\times \Borel(\Hilbert),\mathcal{L}^2(\Proba,\Kilbert),\rho )\\
   &= \bigg\{g(\cdot,\vartheta,\Proba_\vartheta) \in SBC_0(\mathbb{R}, \mathcal{L}^2(\Proba,\Kilbert),\rho):\mbox{ for any } \vartheta \in \mathcal{L}^2(\Proba,\Kilbert), \Proba_\vartheta \in \Borel(\Hilbert))\bigg\}.
 \end{align*}

\begin{definition} An $\mathcal{L}^2$-continuous stochastic process $\vartheta\, :\, \mathbb{R}\to \mathcal{L}^2(\Proba,\Kilbert)$ is square-mean
	weighted pseudo almost automorphic with respect to $\rho\in \mathcal{M}_\infty$ if it can be decomposed as $\vartheta=\vartheta_1+\phi$, where
	$\vartheta_1 \in SAA(\mathbb{R};\mathcal{L}^2(\Proba,\Kilbert))$ and $\phi\in SBC_0(\mathbb{R},\mathcal {L}^2(\Proba,\Kilbert),\rho)$. The  collection  of  all  square-mean weighted  pseudo  almost  automorphic  processes with  respect  to  $\rho$ is
	denoted by $SWPAA(\mathbb{R},\mathcal {L}^2(\Proba,\Kilbert),\rho)$.
\end{definition}
\begin{definition}
Let $\rho\in \mathcal{M}_\infty$ and $F :\mathbb{R}\times \mathcal {L}^2(\Proba,\Kilbert)\times \Borel(\Hilbert) \rightarrow \mathcal {L}^2(\Proba,\Kilbert)$ be stochastic process. $F$ is square-mean weighted pseudo almost automorphic process in $t\in \mathbb{R}$, for each $\vartheta \in \mathcal {L}^2(\Proba,\Kilbert)$ and $\Proba_\vartheta \in \Borel(\Hilbert),$ if it can be decomposed as
$$F = \vartheta_1 + \phi,$$
where $\vartheta_1 \in SAA(\mathbb{R} \times \mathcal{L}^2(\Proba,\Kilbert)\times \Borel(\Hilbert),\mathcal{L}^2(\Proba,\Kilbert) )$ and $\phi \in SBC_0(\mathbb{R}\times \mathcal {L}^2(\Proba,\Kilbert)\times \Borel(\Hilbert),\mathcal{L}^2(\Proba,\Kilbert),\rho )$. The space of all such stochastic processes is denoted by $SWPAA(\mathbb{R}\times \mathcal {L}^2(\Proba,\Kilbert)\times \Borel(\Hilbert),\mathcal{L}^2(\Proba,\Kilbert),\rho ).$
\end{definition}

\begin{definition}
	Let $\vartheta\, :\, \mathbb{R}\to \mathcal{L}^2(\mathbb{R},\Kilbert)$ be a $\mathcal{L}^2$-continuous stochastic process. $\vartheta$ is weighted pseudo automorphic in distribution with respect to $\rho\in \mathcal{M}_\infty$, if it can be decomposed as $\vartheta=\vartheta_1+\phi$, where $\vartheta_1$ is almost automorphic in distribution and $\phi\in SBC_0(\mathbb{R},\mathcal {L}^2(\Proba,\Kilbert),\rho)$.
\end{definition}

\begin{definition}
A set $\Y$ is translation invariant if for any $\vartheta(t)\in \Y,\; \vartheta(t+s)\in \Y$ for any $s\in \mathbb{R}$.
\end{definition}
We denote $\mathcal{M}^{inv}=\{ \rho \in\mathcal{M}_\infty,\; | \; SBC_0(\mathbb{R},\rho) \mbox{ is translation invariant } \}$.

\begin{definition}\cite{Chen}
  An  $\mathcal{L}^2$-continuous  stochastic  process  $f(t, s)  : \mathbb{R} \times \mathbb{R} \to \mathcal{L}^2(\Proba, \Hilbert)$ is square-mean  bi-almost  automorphic  if  for  every  sequence  of  real  numbers  $\{s^\prime_n\}$,  there  exists  a  subsequence $\{s_n\}$ and a continuous function $g : \mathbb{R} \times \mathbb{R} \to \mathcal{L}^2(\Proba, \Hilbert)$ such that
\begin{equation*}
\lim\limits_{n\to \infty}\Ex\| f(t+s_n, s+s_n)-g(t,s)\|^2=0 \mbox{ and }\lim\limits_{n\to \infty}\Ex\| g(t-s_n, s-s_n)-f(t,s)\|^2=0.
\end{equation*}
\end{definition}
 The collection of all square-mean bi-almost automorphic processes is denoted by $SBAA(\mathbb{R} \times \mathbb{R}, \mathcal{L}^2(\Proba, \Hilbert))$.

\subsection{Fractional Brownian motion}\label{mbf}
Let $(\Omega ,\F,P)$ be a complete probability space and consider the two separable Hilbert spaces $\Kilbert$ and $\Kilbert_1$ such that $\Kilbert\hookrightarrow \Kilbert_1$ and the embedding is a Hilbert-Schmidt operator. Let $\mathcal{Q}$ be the trace class operator that is self-adjoint and positive.
\begin{definition}
A $\Kilbert$-valued Gaussian process $\big\{B^H(t),  t\in\mathbb{R}\big\}$ on $(\Omega, \F, \Proba)$ is a \emph{fractional Brownian
 motion of $\mathcal{Q}$-covariance type with Hurst parameter} $H\in (0, 1)$ (or, more simply, a fractional $\mathcal{Q}$-Brownian motion with Hurst parameter $H$) if
\begin{description}
  \item [(1)] $\Ex\big[B^H(t)\big]=0$ for all $t\in\mathbb{R}$,
  \item [(2)] $\operatorname{cov} (B^H(t), B^H(s))
 =\frac{1}{2}\big (|t|^{2H}+|s|^{2H}-|t-s|^{2H}\big) \mathcal{Q}$ \quad for all $t\in\mathbb{R}$,
  \item [(3)] $\big\{B^H(t),  t\in\mathbb{R}\big\}$ has $\Kilbert$-valued, continuous sample
 paths a.s.-$P$,
\end{description}
where $\operatorname{cov}(X, Y)$ denotes the \emph{covariance operator} for the Gaussian random variables $X$ and $Y$ and $\Ex$ stands for the mathematical expectation on $(\Omega, \F, \Proba)$.
\end{definition}

The existence of a fractional $\mathcal{Q}$-Brownian motion is guaranted in the following result.
\begin{theorem} \cite{P}
Let $H\in(0, 1)$ be fixed and $\mathcal{Q}$ be a linear operator such that
$\mathcal{Q}=\mathcal{Q}^{\star}$ and $\mathcal{Q}\in \mathcal{L}_1(\Kilbert)$, where $\mathcal{L}_1(\Kilbert)$ denotes the space of trace class operators on $\Kilbert$. Then, there exists a fractional $\mathcal{Q}$-Brownian motion with Hurst parameter $H$.
\end{theorem}

A fractional Brownian motion of $\mathcal{Q}$-covariance type can be defined directly using the infinite series
\begin{eqnarray}\label{A35}
B^H(t):=\sum_{n=1}^\infty\sqrt{\lambda_n} \beta_n^H(t) e_n,
\end{eqnarray}
where $(e_n,  n\in\mathbb{N})$ is an orthonormal basis in $\Kilbert$ consisting of eigenvectors of $\mathcal{Q}$ and $\{\lambda_n, n\in\mathbb{N}\}$ is the corresponding sequence of eigenvalues of $\mathcal{Q}$ such that $\mathcal{Q}e_n=\lambda_n e_n$ for all $n\in\mathbb{N}$ and $\{\beta_n^H(t),  n\in\mathbb{N},  t\in\mathbb{R}\}$ is a sequence of independent, real-valued standard fractional Brownian motions each with the same Hurst parameter $H\in (0, 1)$. Also, a standard cylindrical fractional Brownian motion in a Hilbert space $\Kilbert$ by is defined by the following formal series
\begin{equation} \label{A36}
B^H(t):=\sum_{n=1}^\infty \beta_n^H(t) e_n\,,
\end{equation}
where $\{e_n,  n\in\mathbb{N}\}$ is a complete orthonormal basis in $\Kilbert$ and $\{\beta_n^H(t),  n\in\mathbb{N},  t\in\mathbb{R}\}$
is a sequence of independent, real-valued standard fractional Brownian motions each with the same Hurst parameter $H\in (0, 1)$. It is well known that the infinite series \eqref{A36} converges in $L^2(\Omega, \Kilbert_1)$, then it defines a $\Kilbert_1$-valued random variable and $\{B^H(t),  t\in\mathbb{R}\}$ is a $\Kilbert_1$-valued fractional Brownian motion of $\mathcal{Q}$-covariance type.\\
Next, in order to define the stochastic integral $\displaystyle\int_{T_1}^{T_2} h(t)dB^H(t)$ for an operator-valued function $h: [T_1, T_2]\to\mathcal{L}(\Kilbert, \Hilbert)$ with $T_1$, $T_2\in\mathbb{R}$, $T_1<T_2$ and for only $H\in (1/2, 1)$, we need the following lemma.\\
\begin{lemma}\cite{P} \label{A40}
If $p>1/H$, then for a given $\varphi\in L^p([T_1, T_2], \mathbb{R})$ the following inequality is true
$$
\int_{T_1}^{T_2}\int_{T_1}^{T_2}\varphi (u)\varphi (v)\phi (u-v)\,du \,dv
\leq C_{T_1, T_2}\big|\varphi\big|^2_{L^p([T_1, T_2]; \mathbb{R})}
$$
for some $C_{T_1, T_2}>0$ that depends only on $T_1$ and $T_2$.
The function $\phi$ is called a fractional kernel and has the following form
\begin{equation} \label{A1}
\phi (u)=H(2H-1) |u|^{2H-2}\quad \mbox{for all}\quad u \in \mathbb{R}\,.
\end{equation}
\end{lemma}
If $\{\beta^H(t),  t\in\mathbb{R}\}$ is a real-valued standard fractional Brownian motion then
$$\Ex(\beta^H(t), \beta^H(s))=\operatorname{cov} (\beta^H(t), \beta^H(s))=\int_{0}^{t}\int_{0}^{s}\phi (r-u)\,du\,dr.$$
Let $\mathcal{E}$ be the family of $\Hilbert$-valued step functions
\begin{align*}
\mathcal{E} = \Big\{&h :  h(s)=\sum_{i=1}^{n-1} h_i \chi_{[t_i, t_{i+1})}(s),
\,T_1=t_1<t_2<\dots<t_n=T_2\\
&~~~~~~~~~~~~~~~~~~~~~~~~~~~~~~~~~~~~\text{and } h_i\in\Hilbert \text{ for }i\in\{1, \ldots, n-1\}\Big\}.
\end{align*}
For $h\in\mathcal{E}$, we define the stochastic integral as follows
\begin{equation} \label{def1}
\int_{T_1}^{T_2} h(t)\,d\beta^H(t):=\sum_{i=1}^{n-1}h_i(\beta^H(t_{i+1})-\beta^H(t_i))\, ,
\end{equation}
where $\{\beta^H(t), t\in [T_1, T_2]\}$ is a scalar fractional Brownian motion.
The expectation of this random variable is zero and the second moment is
\begin{align*}
\Ex\big\|\int_{T_1}^{T_2}h(t)\,d\beta^H(t)\big\|^2_\Hilbert&=\Ex\langle \sum_{i=1}^{n-1}h_i(\beta^H(t_{i+1})-\beta^H(t_i)), \sum_{i=1}^{n-1}h_i(\beta^H(t_{i+1})-\beta^H(t_i))\rangle_\Hilbert\\
&=\sum_{i=1}^{n-1}\sum_{j=1}^{n-1}\langle h_i,h_j\rangle_\Hilbert\Ex\Big[[\beta^H(t_{i+1})-\beta^H(t_i)][\beta^H(t_{j+1})-\beta^H(j_i)]\Big]\\
&=\sum_{i=1}^{n-1}\sum_{j=1}^{n-1}\langle h_i,h_j\rangle_\Hilbert\Big\{\Ex([\beta^H(t_{i+1})][\beta^H(t_{j+1})])-\Ex([\beta^H(t_{i})][\beta^H(t_{j})])\Big\}\\
&=\sum_{i=1}^{n-1}\sum_{j=1}^{n-1}\langle h_i,h_j\rangle_\Hilbert\int_{t_j}^{t_j+1}\int_{t_i}^{t_i+1}\phi (u-v)\,du\,dv\,\\
&=\int_{T_1}^{T_2}\int_{T_1}^{T_2}\sum_{i=1}^{n-1}\sum_{j=1}^{n-1}\langle h_i\chi_{[t_i, t_{i+1})}(u),h_j\chi_{[t_j, t_{j+1})}(v)\rangle_\Hilbert\phi (u-v)\,du\,dv\,\\
&=\int_{T_1}^{T_2}\int_{T_1}^{T_2}\langle h(u), h(v)\rangle_\Hilbert\phi (u-v)\,du\,dv\,.
\end{align*}
Using Lemma \ref{A40} and the fact that $\mathcal{E}$ is dense in $L^p([T_1, T_2], \Hilbert)$, it follows that for $h$ in $L^p([T_1, T_2], \Hilbert)$
$$
\Ex\big\|\int_{T_1}^{T_2}h(t)\,d\beta^H(t)\big\|^2_\Hilbert
\leq C_{T_1, T_2, p}\Big(\int_{T_1}^{T_2}\|h(s)\|^p_\Hilbert\,ds\Big)^{2/p}\,
$$
for some constant $C_{T_1, T_2, p}$ that depends only on $T_1$, $T_2$, and $p$.\\
Now, let $h: [T_1, T_2]\to\mathbb{L}_2$, where $\mathbb{L}_2= L_2(\Kilbert; \Hilbert)$ be the space of all Hilbert-Schmidt operators acting between $\Kilbert$ and $\Hilbert$. We assume that $h(\cdot)x\in L^p([T_1, T_2]; \Hilbert)$ and
\begin{eqnarray}\label{A41}
\int_{T_1}^{T_2}\int_{T_1}^{T_2}\|h(s)\|_{\mathbb{L}_2}\|h(r)\|_{\mathbb{L}_2}\phi (r-s)\,dr ds
<\infty\,
\end{eqnarray}
for any $x\in\Kilbert$ and for an arbitrary $p>1/H$ fixed.\\
Then for a $\Kilbert$-valued standard cylindrical fractional Brownian motion and for $h: [T_1, T_2]\to\mathbb{L}_2$, we define the stochastic integral by
\begin{equation} \label{A43}
\int_{T_1}^{T_2}h(t)\,dB^H(t):=\sum_{n=1}^\infty\int_{T_1}^{T_2} h(t)e_n\,d\beta_n^H(t),
\end{equation}
where $\{e_n,  n\in\mathbb{N}\}$ is a complete orthonormal basis in $\Kilbert$ and
$\{\beta_n^H(t),  n\in\mathbb{N},  t\in\mathbb{R}\}$
is a sequence of independent, real-valued standard fractional Brownian motions each with the same Hurst parameter $H\in (1/2, 1)$. Since $\{\beta_n^H(t),  n\in\mathbb{N},  t\in\mathbb{R}\}$ is a sequence of independent Gaussian random variables and by \eqref{def1}, the sequence of random variables $\left\{\displaystyle\int_{T_1}^{T_2}g(t)e_n\,d\beta_n^H(t), \;n\in\mathbb{N}\right\}$ are clearly mutually independent Gaussian random variables.\\
The second moment of the stochastic integral \eqref{A43} is given by
\begin{align*}
\Ex\big\|\int_{T_1}^{T_2}h(t)\,d\mathbb{B}^H(t)\big\|^2_\Hilbert
&=\Ex\langle\sum_{n=1}^\infty\int_{T_1}^{T_2}h(t)e_n\,d\beta_n^H(t), \sum_{n=1}^\infty \int_{T_1}^{T_2}h(t)e_n\,d\beta_n^H(t)\rangle_\Hilbert\\
&=\Ex\sum_{n=1}^\infty\langle\int_{T_1}^{T_2}h(t)e_n\,d\beta_n^H(t),\int_{T_1}^{T_2}h(t)e_n\,d\beta_n^H(t)\rangle_\Hilbert\\
&~~~~~~~~~~~~~~~~~~+\Ex\sum_{i\neq j}^\infty\langle\int_{T_1}^{T_2}h(t)e_i\,d\beta_i^H(t),\int_{T_1}^{T_2}h(t)e_j\,d\beta_j^H(t)\rangle_\Hilbert\\
&= \sum_{n=1}^\infty\Ex\big\|\int_{T_1}^{T_2}h(t)e_n\,d\beta_n^H(t)\big\|^2_\Hilbert + 0\\
&= \sum_{n=1}^\infty\int_{T_1}^{T_2}\int_{T_1}^{T_2}\langle h(s)e_n,h(r)e_n\rangle_\Hilbert\phi(r-s)\,dr ds\\
&\leq \int_{T_1}^{T_2}\int_{T_1}^{T_2}\|h(s)\|_{\mathbb{L}_2}\|h(r)\|_{\mathbb{L}_2}\phi (r-s)\,dr ds <\infty\,.
\end{align*}
Hence, the stochastic integral \eqref{A43} is a $\Hilbert$-valued Gaussian random variable.\\
For more details, we refer the reader to \cite{P1,P} and the references therein.

\subsection{Evolution families}
Let $(\mathbb{X}, \|\cdot\|)$ be a Banach space and $\mathtt{T}$ be a linear operator on $\mathbb{X}$. Then $\mathsf{Dom}(\mathtt{T})$, $\varrho(\mathtt{T})$, and $\sigma(\mathtt{T})$  stand
respectively for the domain, resolvent set, and spectrum of $C$. Similarly, one sets $R(\lambda, \mathtt{T}) := (\lambda I - \mathtt{T})^{-1}$ for all $\lambda \in \varrho(\mathtt{T})$ where $I$ is the identity operator for $\mathbb{X}$. We denote by ${\mathcal L}(\mathbb{X})$ the space of all bounded linear operators from $\mathbb{X}$ to itself.\\
The following definition was introduced by Acquistapace and Terreni in \cite{AT}.
\begin{definition}\cite{AT}\label{DEF} \rm
A family of closed linear operators $A(t)$ for $t\in \mathbb{R}$ on $\mathbb{X}$ with domain $\mathsf{Dom}(A(t))$ (possibly not densely defined) satisfy the so-called Acquistapace-Terreni condition, if there exist constants $\omega\geq 0$, $\theta
\in (\pi/2,\pi)$, $L,K > 0$, $ a,  b \in (0,1]$ with
$ a +  b > 1$ such that
\begin{equation}\label{AT1}
\\S_\theta \cup \{0\} \subset \varrho(A(t)-\omega) \ni \lambda,\quad
  \|R(\lambda,A(t)-\omega)\|\le \frac{K}{1+|\lambda|} \quad \text{for all }
   t \in \mathbb{R}
\end{equation}
   and
\begin{equation}\label{AT2}
\|(A(t)-\omega)R(\lambda,A(t)-\omega)\,[R(\omega,A(t))-R(\omega,A(s))]\|
  \le L\, \frac{|t-s|^ a}{|\lambda|^{ b}}
  \end{equation}
for $t,s\in\mathbb{R}$, $ \lambda \in S_\theta:=\{\lambda\in\mathbb{C}\setminus\{0\}: |\arg \lambda|\le\theta\}$.
\end{definition}
When $A(t)$ has a constant domain $\mathsf{D}=\mathsf{Dom}(A(t))$ then condition (\ref{AT2}) can be replaced with the following one: There exist constants $L>0$ and $0<\alpha\leq1$ such that
\begin{align}\label{ATbis}
    \|(A(t) - A(s))R(\omega,A(r))\|\leq L\|t-s\|^{\alpha} \quad \mbox{for all}\quad t,s,r \in \mathbb{R}.
\end{align}
More details can be found in \cite{Amann}.
\begin{theorem}\cite{BD}
Let $A(t)$ be a family of closed linear operators which satisfies Acquistapace-Terreni conditions. Then there exists a unique evolution family
$$
\mathcal{U}= \{U(t,s): t, s \in \mathbb{R} \text{ such that } t > s\}
$$
on $\mathbb{X}$ such that

\begin{enumerate}
\item[(a)] $U(t, s)\mathbb{X} \subseteq \mathsf{Dom}(A(t))$ for all $t, s \in \mathbb{R}$ with $t > s$;
\item[(b)]  $U(t,s)U(s,r)=U(t,r)$ for $t,s \in \mathbb{R}$ such that
 $t \geq s \geq r$;

\item[(c)] $U(t,t)=I$ for $t \in \mathbb{R}$ where $I$ is the identity
 operator of $\mathbb{X}$;

\item[(d)] $(t,s)\to U(t,s)\in {\mathcal L}(\mathbb{X})$ is continuous for $t>s$;

\item[(e)] $U(\cdot,s)\in C^1((s,\infty),{\mathcal L}(\mathbb{X}))$, $
\frac{\partial U}{\partial t}(t,s) =A(t)U(t,s)$ and
\begin{align*}\label{au}
  \|A(t)^k U(t,s)\|&\le K\,(t-s)^{-k}
  \end{align*}
for $0< t-s\le 1$ and $k=0,1$.
\end{enumerate}
\end{theorem}
\begin{definition} \rm
An evolution family $\mathcal{U} = \{U(t,s): t, s \in \mathbb{R}
 \text{ such that } t \geq s\}$ is said to have an {\it exponential dichotomy} (or is {\it hyperbolic}) if there are projections $P(t)$ ($t\in\mathbb{R}$) that are uniformly bounded, strongly continuous in $t$ and there are constants $\delta>0$  and $N\ge1$ such that
\begin{enumerate}
\item[(f)] $U(t,s)P(s) = P(t)U(t,s)$;
\item[(g)] the restriction $U_Q(t,s):Q(s)\mathbb{X}\to Q(t)\mathbb{X}$ of $U(t,s)$ is invertible;
\item[(h)] $\|U(t,s)P(s)\| \le Ne^{-\delta (t-s)}$ and $\|\widetilde{U}_Q(s,t)Q(t)\|\le Ne^{-\delta (t-s)}$ for $t\ge s,\quad t,s\in \mathbb{R}$
\end{enumerate}
where $Q(\cdot)=I-P(\cdot)$ and $\widetilde{U}_Q(s,t):=U_Q(t,s)^{-1}$.\\\\
Note that if $U(t,s)$ is exponentially stable, then $\mathcal{U} $ is {\it hyperbolic} which $P(t)=I$.
\end{definition}
More details about the evolutions families can be found in \cite{Lun,EN}.

Throughout this work, we impose the following assumptions:
\begin{enumerate}
	\item The family of operators $A(t)$ on $L^{2}(\Omega,\Hilbert)$ satisfies the Acquistpace-Terreni condition and the evolution family $\mathcal{U}=\big\{U(t, s),   t\ge s\big\}$ associated with $A(t)$ is exponentially stable ( that is there exist constant $M$, $\delta >0$ such that
	$$\|U(t, s)\|\leq M e^{-\delta(t-s)} \quad \mbox{for all}\quad t\geq s).$$
	This implies $U(t,s)$ is {\it hyperbolic} whith $P(t)=I$.
	\item $U(t,s)z \in  SBAA(\mathbb{R},{\mathcal L}(L^{2}(\Omega,\Hilbert)))$  uniformly for all $z$ in any bounded subset in $L^{2}(\Omega,\Hilbert)$.
\end{enumerate}
If the above two conditions hold, we say that condition $\mathbf{(H_0)}$ holds.

We recall the following lemma that will be crucial in the proof of the main result.
\begin{lemma}\cite{Kamenskii}\label{Gronw}
	Let $g : \mathbb{R}\to \mathbb{R}$ be a continuous function such that, for every $t \in \mathbb{R}$,
	\begin{equation}\label{Gronw1}
	0\le g(t)\le \alpha(t)+\beta_1\displaystyle\int_{-\infty}^{t}e^{-\delta_1(t-s)}g(s)ds+\cdots+\beta_n\displaystyle\int_{-\infty}^{t}e^{-\delta_n(t-s)}g(s)ds
	\end{equation}
	for some locally integrable function $\alpha : \mathbb{R} \to \mathbb{R}$, and for some constants $\beta_1,\cdots,\beta_n>0$, and $\delta_1,\cdots,\delta_n>\beta$, where $\beta:=\displaystyle\sum_{i=1}^{n} \beta_i$. We assume that the integrals on the right side of \eqref{Gronw1} are convergent. Let $\delta_{\max}=\displaystyle\min_{1\le i\le n}\delta_i$. Then, for every $ \gamma\in (0,\delta_{\max}-\beta)$
	such that $\displaystyle\int_{-\infty}^{0}e^{\gamma\,s}\alpha(s)ds$
	converges, we have, for every $t\in \mathbb{R}$,
	$$ g(t)\le \alpha(t)+\beta\displaystyle\int_{-\infty}^{t}e^{-\gamma(t-s)}\alpha(s)ds,
	$$for every $t\in\mathbb{R}.$\\
	In particular, if $\alpha(t)$ is constant, then $$ g(t)\le \alpha\dfrac{\delta_{\max}}{\delta_{\max}-\beta},$$
for every $t\in\mathbb{R}.$
\end{lemma}

Let $\mathbb{U}$ and $\mathbb{U}_0$ be real separable Hilbert spaces, $\mathbb{L}_2:= L_2(\mathbb{U}; \Hilbert)$ denote the space of all Hilbert-Schmidt operators acting between $\mathbb{U}$ and $\Hilbert$ equipped with the Hilbert-Schmidt norm $\|\cdot\|_{\mathbb{L}_2}$ and $\mathbb{L}_2^0=L_2(\mathbb{U}_0; \Hilbert)$.

For each $t\in\mathbb{R}$, we denote by $\F_t$ the $\sigma$-field generated by the random variables $\big\{B^H(s),  W(s), s\leq t \big\}$ and the $\mathbb{P}$-null sets. In addition to the natural filtration $\big\{\F_t,  t\in\mathbb{R}\big\}$, we consider a larger filtration $\big\{\mathcal{G}_t,  t\in\mathbb{R}\big\}$ for which
\begin{itemize}
	\item [(1)] $\{\mathcal{G}_t\}$ is right-continuous and $\mathcal{G}_0$ contains the P-null sets,
	\item [(2)] $B^H$ is $\mathcal{G}_0$-measurable and $W$ is a $\mathcal{G}_t$-Brownian motion.
\end{itemize}

To prove Theorem \ref{Th1}, we need the following lemma that is a particular case of Lemma 2.2 in \cite{Se}.
\begin{lemma} \label{Bridge}
	Let $\Gamma :[0,T]\times \Omega \rightarrow {\mathcal L}(L^{p}(\Omega,\Hilbert)) $ be an $\F_{t}-$adapted measurable stochastic process satisfying $\displaystyle\int_{0}^{T}\Ex\big\|\Gamma(t)\big\|^{2}dt<\infty$ almost surely. Then, for any $p \in [1,\infty[$, there exits a constant $\widetilde{C_{p}} > 0$ such that
	\begin{align*}
	\Ex\left( \sup_{0\leq t \leq T}\Big\|\int_{0}^{T} \Gamma(s)dW(s)\Big\|^{p}\right)\leq
	\widetilde{C_{p}}\Big(\int_{0}^{T}\Ex\Big\|\Gamma(s)\Big\|^{2}ds \Big)^{\frac{p}{2}}
	\end{align*}
	for $T >0$.
\end{lemma}

\section{Almost automorphic mild solution for Equ.\eqref{C1}}
We study the existence of almost automorphic mild solution for the mean field system \eqref{C1} in this section. First, we give the definition mild solution.
\begin{definition}
 A $\mathcal{G}_t$-progressively measurable process $\{\vartheta(t)\}_{t\in\mathbb{R}}$ is a mild solution of Equ.\eqref{C1} if it satisfies the stochastic integral equation
 \begin{equation}
 \begin{array}{ll}
\vartheta(t)=& U(t,b)\vartheta(b) +\displaystyle\int_{b}^{t} U(t,s)f(s,\vartheta(s),\Proba_{\vartheta(s)})ds + \displaystyle\int_{b}^{t} U(t,s)\theta(s,\vartheta(s),\Proba_{\vartheta(s)})dW(s)\\\\
 & + \displaystyle\int_{b}^{t} U(t,s)\psi(s,\Proba_{\vartheta(s)})dB^H(s)
 \end{array}
 \end{equation}
 for all $t\ge b$ and for each $b\in \mathbb{R}$.
\end{definition}
We introduce the following hypotheses which are assumed hereafter :
\begin{description}
\item[$\mathbf{(H_1)}$] The functions $f: \mathbb{R}\times \Hilbert\times \Borel(\Hilbert)\to \Hilbert$, $\theta : \mathbb{R}\times \Hilbert\times \Borel(\Hilbert)\to \mathbb{L}_2^0$ and $\Psi : \mathbb{R}\times \Borel(\Hilbert)\to \mathbb{L}_2$ are square-mean almost automorphic in $t\in \mathbb{R}$, for each $\vartheta\in \mathcal{L}^2(\Proba,\Hilbert)$.
\item[$\mathbf{(H_2)}$] There exist a constant $\mathbf{K}>0$ such that
\begin{align*}
&\|f(t,x,\nu_1)-f(t,y,\nu_2\|^2\le \mathbf{K} \left(\|x-y\|^2 + \mathcal{W}^2(\nu_1,\nu_2 \right),\\
&\|\theta(t,x,\nu_1)-\theta(t,y,\nu_2)\|^2_{\mathbb{L}_2^0}\le \mathbf{K} \left(\|x-y\|^2 + \mathcal{W}^2(\nu_1,\nu_2 \right),\\
& \|\psi(t,\nu_1)-\psi(t,\nu_2)\|_{\mathbb{L}_2}\le \mathbf{K}\;  \mathcal{W}(\nu_1,\nu_2),
 \end{align*}
for all $x,y\in \Hilbert$, $\nu_1,\nu_2\in \Borel(\Hilbert)$ and $t\in \mathbb{R}$. \end{description}

\begin{theorem}\label{Th1} Assume $\mathbf{(H_{0})}$, $\mathbf{(H_1)}$ and $\mathbf{(H_2)}$  hold. Then, Equ.\eqref{C1} has a unique $\mathcal{L}^2$-bounded solution provided that
	\begin{equation}\label{Cond1}
		2\mathbf{K}M^2\left(\dfrac{1}{\delta^2}+\dfrac{\widetilde{C_2}}{2\delta}\right)<1
	\end{equation} and
	\begin{equation}\label{Cond1+}
\dfrac{\beta_2}{\delta}\left[1+\,\dfrac{\beta_2}{\delta}\right]<1,
	\end{equation}
	where $\beta_2$ is a positive constant (see \eqref{beta12}).\\	Furthermore, this unique $\mathcal{L}^2$-bounded solution is almost automorphic in distribution, provided that
	\begin{equation}\label{Cond2}
	\left(\dfrac{18\mathbf{K}\,M^2}{\delta}\,+18\,\widetilde{C_2}\,M^2\, \mathbf{K}+9H(2H-1)M^2 (\mathbf{K})^2\right)<1.
	\end{equation}
\end{theorem}
\begin{proof}
Let $\mu \in\mathcal{C}_b(\mathbb{R},\Borel(\Hilbert))$ be fixed and consider the operator  $\Upsilon $ defined by
\begin{equation*}
\begin{array}{ll}
(\Upsilon\vartheta)(t)=&\displaystyle\int_{-\infty}^{t} U(t,s)f(s,\vartheta(s),\mu(s))ds + \displaystyle\int_{-\infty}^{t} U(t,s)\theta(s,\vartheta(s),\mu(s))dW(s)\\\\
& + \displaystyle\int_{-\infty}^{t} U(t,s)\psi(s,\mu(s))dB^H(s).
\end{array}
\end{equation*}
for $ \vartheta \in \mathcal{C}_b(\mathbb{R},\mathcal{L}^2(\Proba,\Hilbert))$.
We break the proof into a sequence steps.\\
\textbf{Step 1}. Let us check that $\Upsilon\vartheta$ belongs to $ \mathcal{C}_b(\mathbb{R},\mathcal{L}^2(\Proba,\Hilbert))$.\\
 For arbitrary $t\ge t_1$,
 \begin{equation*}
 \begin{array}{rl}
 &\Ex\|(\Upsilon\vartheta)(t)-(\Upsilon\vartheta)(t_1)\|^2\\\\
 =&\Ex\Big\| \displaystyle\int_{-\infty}^{t} U(t,s)f(s,\vartheta(s),\mu(s))ds
 + \displaystyle\int_{-\infty}^{t} U(t,s)\theta(s,\vartheta(s),\mu(s))dW(s)\\\\
 & + \displaystyle\int_{-\infty}^{t} U(t,s)\psi(s,\mu(s))dB^H(s)-\displaystyle\int_{-\infty}^{t_1} U(t_1,s)f(s,\vartheta(s),\mu(s))ds\\\\
 &- \displaystyle\int_{-\infty}^{t_1} U(t_1,s)\theta(s,\vartheta(s),\mu(s))dW(s) - \displaystyle\int_{-\infty}^{t_1} U(t_1,s)\psi(s,\mu(s))dB^H(s)\Big\|^2\\\\
  =&3\mathbb{E}\Big\| \displaystyle\int_{-\infty}^{t} U(t,s)f(s,\vartheta(s),\mu(s))ds
-\displaystyle\int_{-\infty}^{t_1} U(t_1,s)f(s,\vartheta(s),\mu(s))ds\Big\|^2\\\\
& + 3\mathbb{E}\Big\|\displaystyle\int_{-\infty}^{t} U(t,s)\theta(s,\vartheta(s),\mu(s))dW(s)- \displaystyle\int_{-\infty}^{t_1} U(t_1,s)\theta(s,\vartheta(s),\mu(s))dW(s)\Big\|^2\\\\
 & + 3\mathbb{E}\Big\|\displaystyle\int_{-\infty}^{t} U(t,s)\psi(s,\mu(s))dB^H(s) - \displaystyle\int_{-\infty}^{t_1} U(t_1,s)\psi(s,\mu(s))dB^H(s)\Big\|^2
 \\\\
=&3\,\left[P_1(t)+P_2(t)+P_3(t)\right].
 \end{array}
 \end{equation*}
For $P_1(t)$, it follows from the H\"{o}lder inequality and exponential dissipation property of $U(t,s)$ that
\begin{equation}\label{EstimationContinuity1}
\begin{array}{rl}
&\Ex\Big\| \displaystyle\int_{-\infty}^{t} U(t,s)f(s,\vartheta (s),\mu(s))ds
-\displaystyle\int_{-\infty}^{t_1} U(t_1,s)f(s,\vartheta(s),\mu(s))ds\Big\|^2\\
&\le2\mathbb{E}\Big\| \displaystyle\int_{-\infty}^{t_1} \left(U(t,s)-U(t_1,s)\right)f(s,\vartheta(s),\mu(s))ds\Big\|^2+ 2\mathbb{E}\Big\|\displaystyle\int_{t_1}^{t} U(t,s)f(s,\vartheta(s),\mu(s))ds
\Big\|^2\\
&\le2\mathbb{E}\Big\| \displaystyle\int_{-\infty}^{t_1} \left[U(t,t_1)- I\right]U(t_1,s)f(s,\vartheta (s),\mu(s) )ds\Big\|^2\\
&\qquad + 2\left(\displaystyle\int_{t_1}^{t} M^2e^{-2\,\delta\,(t-s)}ds\right) \left(\displaystyle\int_{t_1}^{t}\Ex\|f(s,\vartheta(s),\mu(s))\|^2\,ds\right).\\
&\le2M^{2}\left(\displaystyle\int_{t_1}^{t} e^{-\delta\,(t_{1}-s)}ds\right) \left(\displaystyle\int_{-\infty}^{t_1}e^{-\delta\,(t_{1}-s)} \Ex\Big\| \left[U(t,t_1)- I\right]f(s,\vartheta (s),\mu(s) )\Big\|^2ds\right)\\
&\qquad + 2\left(\displaystyle\int_{t_1}^{t} M^2e^{-2\,\delta\,(t-s)}ds\right) \left(\displaystyle\int_{t_1}^{t}\Ex\|f(s,\vartheta(s),\mu(s))\|^2\,ds\right).\\
&\qquad + 2\left(\displaystyle\int_{t_1}^{t} M^2e^{-2\,\delta\,(t-s)}ds\right) \left(\displaystyle\int_{t_1}^{t}\Ex\|f(s,\vartheta(s),\mu(s))\|^2\,ds\right).\\
&\le\dfrac{2M^{2}}{\delta} \left(\displaystyle\int_{-\infty}^{t_1}e^{-\delta\,(t_{1}-s)} \Ex\Big\| \left[U(t,t_1)- I\right]f(s,\vartheta (s),\mu(s) )\Big\|^2ds\right)\\
&\qquad + 2M^2\sup_{s\in \mathbb{R}}\Ex\|f(s,\vartheta(s),\mu(s))\|^2\left(t-t_1\right)^{2}.\\
\end{array}
\end{equation}
From the strong continuity and exponential dissipation property of $U(t,t_{1})$, for $t_{1}\in (-\infty, t]$, we get
\begin{equation*}
\Ex\|\left[U(t,t_1)- I\right]f(s,\vartheta(s),\mu(s))\|^2\to 0
\end{equation*}
 as $t\to t_1$. As $t$ is in the neighborhood of $t_1$ sufficiently small, we have  \begin{equation*}
 \begin{array}{rl}
e^{-\delta\,(t_{1}-s)}\Ex\|\left[U(t,t_1)- I\right]f(s,\vartheta(s),\mu(s))\|^2
\le (M^2+1)e^{ \delta(t_1-s)} \Ex\|f(s,\vartheta(s),\mu(s))\|^2.
 \end{array}
 \end{equation*}
Since
\begin{equation*}
 \displaystyle\int_{-\infty}^{t_1}(M^2+1)e^{- \delta(t_1-s)} \Ex\|f(s,\vartheta(s),\mu(s))\|^2\,ds<\infty.
\end{equation*}
The Lebesgue dominated convergence theorem implies that
 \begin{equation*}
 \left(\displaystyle\int_{-\infty}^{t_1}e^{-\delta\,(t_{1}-s)} \Ex\Big\| \left[U(t,t_1)- I\right]f(s,\vartheta (s),\mu(s) )\Big\|^2ds\right)\to 0 \mbox{ as } t\to t_1.
 \end{equation*}
Hence, from \eqref{EstimationContinuity1}, it follows that

 \begin{equation}\label{EP1}
P_1(t)\longrightarrow0\quad \mbox{ as }\quad t\to t_1.
 \end{equation}

The case $t\le t_1 $ can be argued similarly.\\ Let
$\widetilde{W}(\tau)=W(\tau +t-t_1)-W(t-t_1)$ for each $\tau\in \mathbb{R}$. We know that $\widetilde{W}$ is Wiener process and has the same law of $W(\tau)$. Putting $F_{\theta}(s)=\theta(s,\vartheta(s),\mu(s))$, letting $\tau=s-t+t_1,$ and using  Lemma \ref{Bridge}, yileds
\begin{equation*}
\begin{array}{rl}
&\Ex\Big\|\displaystyle\int_{-\infty}^{t} U(t,s)\theta(s,\vartheta(s),\mu(s))dW(s)- \displaystyle\int_{-\infty}^{t_1} U(t_1,s)\theta(s,\vartheta(s),\mu(s))dW(s)\Big\|^2\\\\
=&\Ex\Big\|\displaystyle\int_{-\infty}^{t} U(t,s)F_{\theta}(s) dW(s)- \displaystyle\int_{-\infty}^{t_1} U(t_1,s)\,F_{\theta}(s)\,dW(s)\Big\|^2\\\\
=&\Ex\Big\|\displaystyle\int_{-\infty}^{t_1} U(t,t-t_1+s)\,F_{\theta}(t-t_1+s)\,dW(t-t_1+s)- \displaystyle\int_{-\infty}^{t_1} U(t_1,s)\, F_{\theta}(s)\,dW(s)\Big\|^2\\\\
=&\Ex\Big\|\displaystyle\int_{-\infty}^{t_1} U(t,t-t_1+s)\,F_{\theta}(t-t_1+s)\,d\widetilde{W}(s)- \displaystyle\int_{-\infty}^{t_1} U(t_1,s)\, F_{\theta}(s)\,d\widetilde{{W}}(s)\Big\|^2\\\\

\end{array}
\end{equation*}
\begin{equation*}
\begin{array}{rl}
=&\Ex\Big\|\displaystyle\int_{-\infty}^{t_1} \left[U(t,t-t_1+s)\,F_{\theta}(t-t_1+s)-U(t_1,s)\, F_{\theta}(s)\,\right]d\widetilde{W}(s)\Big\|^2\\\\
=&\Ex\Big\|\displaystyle\int_{-\infty}^{t_1} \big[U(t,t-t_1+s)(\,F_{\theta}(t-t_1+s)-F_{\theta}(s))\\\\
&\qquad+ U(t,t-t_1+s)F_{\theta}(s)-U(t_1,s)\, F_{\theta}(s)\,\big]d\widetilde{W}(s)\Big\|^2\\\\
\le&2\widetilde{C_2}\displaystyle\int_{-\infty}^{t_1}\Big[ \Ex\big\|U(t,t-t_1+s)\big(\,F_{\theta}(t-t_1+s)-F_{\theta}(s)\big)\big\|^2_{\mathbb{L}^0_2}\\\\
&\qquad+ \Ex\big\|U(t,t-t_1+s)F_{\theta}(s)-U(t_1,s)\, F_{\theta}(s)\,\big\|^2_{\mathbb{L}^0_2}\Big]\,ds\\\\
\le&2\widetilde{C_2}\displaystyle\int_{-\infty}^{t_1}\Big[\|U(t,t-t_1+s)\|^2\, \Ex\big\|F_{\theta}(t-t_1+s)-F_{\theta}(s)\big\|^2_{\mathbb{L}^0_2}\\\\
&\qquad+ \Ex\big\|U(t,t-t_1+s)F_{\theta}(s)-U(t_1,s)\, F_{\theta}(s)\,\big\|^2_{\mathbb{L}^0_2}\Big]\,ds\\\\
\le&2\widetilde{C_2}M^2\displaystyle\int_{-\infty}^{t_1}\, e^{-2\,\delta\, (t_1-s)}\, \Ex\big\|F_{\theta}(t-t_1+s)-F_{\theta}(s)\big\|^2_{\mathbb{L}^0_2}\,ds\\\\
&\qquad+ 2\widetilde{C_2}\displaystyle\int_{-\infty}^{t_1}\Ex\big\|U(t,t-t_1+s)F_{\theta}(s)-U(t_1,s)\, F_{\theta}(s)\,\big\|^2_{\mathbb{L}^0_2}\,ds.
\end{array}
\end{equation*}

Since \begin{equation*}
\, e^{-2\,\delta\, (t_1-s)}\, \Ex\big\|F_{\theta}(t-t_1+s)-F_{\theta}(s)\big\|^2_{\mathbb{L}^0_2}\to 0 \mbox{ as }t\to t_1
\end{equation*}
and
\begin{equation*}
e^{-2\,\delta\, (t_1-s)}\, \Ex\big\|F_{\theta}(t-t_1+s)-F_{\theta}(s)\big\|_{\mathbb{L}^0_2}^2\le K_1\,e^{-2\,\delta\, (t_1-s)},
\end{equation*}
where $K_1$ is a positive constant related to the boundedness of $F_{\theta}(s)$, it follows by Lebesgue dominated convergence theorem that
\begin{equation}\label{LB1}
\displaystyle\int_{-\infty}^{t_1}\, e^{-2\,\delta\, (t_1-s)}\, \Ex\big\|F_{\theta}(t-t_1+s)-F_{\theta}(s)\big\|^2_{\mathbb{L}^0_2}\,ds\to 0 \mbox{ as } t\to t_1.
\end{equation}
From the strong continuity of $U(t,s)$, we have
\begin{equation*}
\Ex\big\|U(t,t-t_1+s)F_{\theta}(s)-U(t_1,s)\, F_{\theta}(s)\,\big\|^2_{\mathbb{L}^0_2}\to \mbox{ as } t\to t_1,
\end{equation*}
and from the exponential dissipation property of $U(t,s)$, we have
\begin{equation*}
\Ex\big\|U(t,t-t_1+s)F_{\theta}(s)-U(t_1,s)\, F_{\theta}(s)\,\big\|^2_{\mathbb{L}^0_2}\le  4M^2\, e^{-2\delta\,(t_1-s)}\Ex\|F_{\theta}(s)\|^2_{\mathbb{L}^0_2}.
\end{equation*}
From the stochastic boundedness of $F_{\theta}\in SBC_0(\mathbb{R},\rho)$ and Lebesgue dominated convergence theorem, we deduce that
\begin{equation}\label{LB2}
\displaystyle\int_{-\infty}^{t_1}\Ex\big\|U(t,t-t_1+s)F_{\theta}(s)-U(t_1,s)\, F_{\theta}(s)\,\big\|^2_{\mathbb{L}^0_2}\,ds\to 0 \mbox{ as } t\to t_1.
\end{equation}
Hence, from \eqref{LB1} and \eqref{LB2}, we obtain
 \begin{equation}\label{EP2}
P_2(t) \longrightarrow 0 \mbox{ as } t\to t_1.
 \end{equation}

 For $P_3(t)=\Ex\big\|\displaystyle\int_{-\infty}^{t} U(t,s) F_{\psi}(s)\,dB^H(s) - \displaystyle\int_{-\infty}^{t_1} U(t_1,s)F_{\psi}(s)\,dB^H(s)\big\|^2$ where $F_{\psi}(s)=\psi(s,\mu(s))$.   Let
 $\widetilde{B^H}(\tau)=B^H(\tau +t-t_1)-B^H(t-t_1)$ for each $\tau\in \mathbb{R}$. We know that $\widetilde{B^H}(\tau)$ is fractional Brownian motion and has the same law of $B^H(\tau)$.

 \begin{equation*}
\begin{array}{rl}
P_3(t)=&\Ex\big\|\displaystyle\int_{-\infty}^{t} U(t,s) F_{\psi}(s)\,dB^H(s) - \displaystyle\int_{-\infty}^{t_1} U(t_1,s)F_{\psi}(s)\,dB^H(s)\big\|^2\\\\
=&\Ex\big\|\displaystyle\int_{-\infty}^{t_1} \big[U(t,t-t_1+s)(\,F_{\theta}(t-t_1+s)-F_{\theta}(s))\\\\
&\qquad\qquad+\quad U(t,t-t_1+s)F_{\theta}(s)-U(t_1,s)\, F_{\theta}(s)\,\big]d\widetilde{B^H}(s)\big\|^2\\\\
\le &H(2H-1)\displaystyle\int_{-\infty}^{t_1}\int_{-\infty}^{t_1} \big\|U(t,t-t_1+s)(\,F_{\theta}(t-t_1+s)-F_{\theta}(s))\\\\
&\qquad\qquad+\quad U(t,t-t_1+s)F_{\theta}(s)-U(t_1,s)\, F_{\theta}(s)\,\big\|_{\mathbb{L}_2}\\\\
&\times \big\|U(t,t-t_1+r)(\,F_{\theta}(t-t_1+r)-F_{\theta}(r))\\\\
&\qquad\qquad+\quad U(t,t-t_1+r)F_{\theta}(r)-U(t_1,r)\, F_{\theta}(r)\,\big\|_{\mathbb{L}_2}\,|r-s|^{2H-2}\,ds\,dr\\\\
\le&H(2H-1)\displaystyle\int_{0}^{+\infty}\int_{0}^{+\infty} \big\|U(t,t-v)(\,F_{\theta}(t-v)-F_{\theta}(v-t_1))\\\\
&\qquad\qquad+\quad U(t,t-v)F_{\theta}(t_1-v)-U(t_1,t_1-v)\, F_{\theta}(t_1-v)\,\big\|_{\mathbb{L}_2}\\\\
&\times \big\|U(t,t-u)(\,F_{\theta}(t-u)-F_{\theta}(u-t_1))\\\\
&\qquad+\quad U(t,t-u)F_{\theta}(u-t_1)-U(t_1,u-t_1)\, F_{\theta}(u-t_1)\,\big\|_{\mathbb{L}_2}\,|v-u|^{2H-2}\,dv\,du
\end{array}
\end{equation*}

\begin{equation*}
\begin{array}{rl}

\le&H(2H-1)\displaystyle\Big(\int_{0}^{+\infty} \big\|U(t,t-v)(\,F_{\theta}(t-v)-F_{\theta}(v-t_1))\\\\
&\qquad\qquad+\quad U(t,t-v)F_{\theta}(t_1-v)-U(t_1,t_1-v)\, F_{\theta}(t_1-v)\,\big\|_{\mathbb{L}_2}\Big)\\\\
&\times \displaystyle\int_{0}^{+\infty}\Big(\big\|U(t,t-u)(\,F_{\theta}(t-u)-F_{\theta}(u-t_1))\\\\
&\qquad+\quad U(t,t-u)F_{\theta}(u-t_1)-U(t_1,u-t_1)\, F_{\theta}(u-t_1)\,\big\|^2_{\mathbb{L}_2}\Big)^{1/2}\,|v-u|^{2H-2}\,dv\,du\\\\

\le&H(2H-1)\displaystyle\Big(\int_{0}^{+\infty} \big\|U(t,t-v)(\,F_{\theta}(t-v)-F_{\theta}(v-t_1))\\\\
&\qquad\qquad+\quad U(t,t-v)F_{\theta}(t_1-v)-U(t_1,t_1-v)\, F_{\theta}(t_1-v)\,\big\|_{\mathbb{L}_2}\Big)\\\\
&\times \displaystyle\int_{0}^{+\infty}\Big(\big\|U(t,t-u)(\,F_{\theta}(t-u)-F_{\theta}(u-t_1))\\\\
&\qquad\qquad+\quad U(t,t-u)F_{\theta}(u-t_1)-U(t_1,u-t_1)\, F_{\theta}(u-t_1)\,\big\|^2_{\mathbb{L}_2}\, du\Big)^{1/2}\\\\
&\qquad\qquad \times\Big(\displaystyle\int_{0}^{+\infty}\,|v-u|^{4H-4}\,du \Big)^{1/2}\,dv\\\\
\le&H(2H-1)\displaystyle\Big(\int_{0}^{+\infty} \big\|U(t,t-v)(\,F_{\theta}(t-v)-F_{\theta}(v-t_1))\\\\
&\qquad\qquad+\quad U(t,t-v)F_{\theta}(t_1-v)-U(t_1,t_1-v)\, F_{\theta}(t_1-v)\,\big\|^2_{\mathbb{L}_2}\, dv\Big)^{1/2}\\\\
&\times \displaystyle\int_{0}^{+\infty}\Big(\big\|U(t,t-u)(\,F_{\theta}(t-u)-F_{\theta}(u-t_1))\\\\
&\qquad\qquad+\quad U(t,t-u)F_{\theta}(u-t_1)-U(t_1,u-t_1)\, F_{\theta}(u-t_1)\,\big\|^2_{\mathbb{L}_2}\, du\Big)^{1/2}\\\\
&\qquad\qquad \times\Big(\displaystyle\int_{0}^{+\infty}\int_{0}^{+\infty}\,|v-u|^{4H-4}\,du\,dv \Big)^{1/2}\\\\
 \end{array}
 \end{equation*}

 \begin{equation*}
\begin{array}{rl}
\le&H(2H-1)\displaystyle\Big(\int_{0}^{+\infty} \big\|U(t,t-v)(\,F_{\theta}(t-v)-F_{\theta}(v-t_1))\\\\
&\qquad\qquad+\quad U(t,t-v)F_{\theta}(t_1-v)-U(t_1,t_1-v)\, F_{\theta}(t_1-v)\,\big\|^2_{\mathbb{L}_2}\, dv\Big)\\\\
&\qquad\qquad \times\Big(\displaystyle\int_{0}^{+\infty}\int_{0}^{+\infty}\,|v-u|^{4H-4}\,du\,dv \Big)^{1/2}\\\\

\end{array}
\end{equation*}
 \begin{equation*}
 \begin{array}{rl}
P_3(t)\le&H(2H-1)\displaystyle\Big(\int_{0}^{+\infty} \big\|U(t,t-v)(\,F_{\psi}(t-v)-F_{\psi}(v-t_1))\\\\
 &\qquad\qquad+\quad U(t,t-v)F_{\psi}(t_1-v)-U(t_1,t_1-v)\, F_{\psi}(t_1-v)\,\big\|^2_{\mathbb{L}_2}\, dv\Big)\\\\
 &\qquad\qquad \times\Big(\displaystyle\int_{0}^{+\infty}\int_{0}^{+\infty}\,|v-u|^{4H-4}\,du\,dv \Big)^{1/2}\\\\
 \le&H(2H-1)\displaystyle\Big(\int_{-\infty}^{t_1} \big\|U(t,t-t_1+s)(\,F_{\psi}(t-t_1+s)-F_{\psi}(s))\\\\
 &\qquad\qquad+\quad U(t,t-t_1+s)F_{\psi}(s)-U(t_1,s)\, F_{\psi}(s)\,\big\|^2_{\mathbb{L}_2}\Big)\\\\
 &\qquad\qquad \times\Big(\displaystyle\int_{0}^{+\infty}\int_{0}^{+\infty}\,|v-u|^{4H-4}\,du\,dv \Big)^{1/2}\\\\
 \le&H(2H-1)\Big(\displaystyle\int_{0}^{+\infty}\int_{0}^{+\infty}\,|v-u|^{4H-4}\,du\,dv \Big)^{1/2}\\\\
 &\qquad \times \displaystyle\Big(\int_{-\infty}^{t_1} \big\|U(t,t-t_1+s)(\,F_{\psi}(t-t_1+s)-F_{\psi}(s))\\\\
 &\hspace{3cm}+\quad U(t,t-t_1+s)F_{\psi}(s)-U(t_1,s)\, F_{\psi}(s)\,\big\|^2_{\mathbb{L}_2} ds\Big)
 \end{array}
 \end{equation*}
Arguing as in \eqref{LB1} and \eqref{LB2}, using  the strong continuity of $U(t,s)$, the exponential dissipation property of $U(t,s)$, the  boundedness of $F_{\psi}\in SBC_0(\mathbb{R},\rho)$ and Lebesgue dominated convergence theorem, we conclude that
 \begin{equation*}
 \displaystyle\int_{-\infty}^{t_1} \big\|U(t,t-t_1+s)(\,F_{\psi}(t-t_1+s)-F_{\psi}(s))+ U(t,t-t_1+s)F_{\psi}(s)-U(t_1,s)\, F_{\psi}(s)\,\big\|^2_{\mathbb{L}_2} ds\to 0,
 \end{equation*}
 as $t\to t_1$. Hence
 \begin{equation}\label{EP3}
 	P_3(t)\to 0 \mbox{ as } t\to t_1.
 \end{equation}

Hence from \eqref{EP1},\eqref{EP2} and \eqref{EP3}, it follows that
\begin{equation*}
\Ex\|(\Upsilon\vartheta)(t)-(\Upsilon\vartheta)(t_1)\|^2 \to 0 \mbox{ as } t\to t_1.
\end{equation*}
 which implies that $\Upsilon\vartheta$ is $\mathcal{L}^2$-continuous.\\
From
$f\in SBC_0(\mathbb{R}\times \mathbb{H}\times \Borel(\Hilbert),\mathcal{L}^2(\Proba,\Kilbert),\rho )$, $\theta\in SBC_0(\mathbb{R}\times \mathcal {L}^2(\Proba,\Kilbert)\times \Borel(\Hilbert),\mathcal{L}^2(\Proba,\mathbb{L}_2^0),\rho )$, $\psi \in SBC_0(\mathbb{R}\times \Borel(\Hilbert),\mathbb{L}_2,\rho )$  and the exponential dissipation property of $U(t,s)$, we  know that $\Upsilon\vartheta$ is $\mathcal{L}^2$-bounded. \\

Next, we shall prove that the operator $\Upsilon$ is a contraction mapping on $\mathcal{C}_b(\mathbb{R},\mathcal{L}^2(\Proba,\Hilbert))$ in square mean.
Let $\vartheta_1,\vartheta_2\in \mathcal{C}_b(\mathbb{R},\mathcal{L}^2(\Proba,\Hilbert))$, $\mu \in \Borel(\Hilbert)$ and $t\in \mathbb{R}$. Observe that
\begin{equation}
\begin{array}{rl}
&\Ex\|(\Upsilon\vartheta_1)(t)-(\Upsilon\vartheta_2)(t)\|^2\\\\
&\le2\Ex\left\|\displaystyle\int_{-\infty}^{t} U(t,s)[f(s,\vartheta_1(s),\mu(s))-f(s,\vartheta_2(s),\mu(s))]ds\right\|^2\\\\
&\quad + 2\Ex\left\|\displaystyle\int_{-\infty}^{t} U(t,s)[\theta(s,\vartheta_1(s),\mu(s))-\theta(s,\vartheta_2(s),\mu(s))]dW(s)\right\|^2\\\\\
&\le  2P_1 + 2P_2,
\end{array}
\end{equation}
where
\begin{equation}
\begin{array}{rl}
P_1 =&\Ex\left\|\displaystyle\int_{-\infty}^{t} U(t,s)[f(s,\vartheta_1(s),\mu(s))-f(s,\vartheta_2(s),\mu(s))]ds\right\|^2\\\\
P_2 =&\Ex\left\|\displaystyle\int_{-\infty}^{t} U(t,s)[\theta(s,\vartheta_1(s),\mu(s))-\theta(s,\vartheta_2(s),\mu(s))]dW(s)\right\|^2.
\end{array}
\end{equation}
From the hypotheses $\mathbf{(H_0)}$ and $\mathbf{(H_2)}$ and using the Cauchy-Schwarz inequality, we get
\begin{equation}\label{Estimation of P1}
\begin{array}{rl}
P_1 \le & \Ex\left\|\displaystyle\int_{-\infty}^{t} U(t,s)[f(s,\vartheta_1(s),\mu(s))-f(s,\vartheta_2(s),\mu(s))]ds\right\|^2\\\\
\le& M^{2}\Ex\left( \displaystyle\int_{-\infty}^{t} e^{-\delta(t-s)}\|f(s,\vartheta_1(s),\mu(s))-f(s,\vartheta_2(s),\mu(s))\|\,ds \right)^2\\\\
\le& M^2\left( \displaystyle\int_{-\infty}^{t}  e^{-\delta(t-s)}ds\right).\; \left(\displaystyle\int_{-\infty}^{t}e^{-\delta(t-s)}\Ex\|f(s,\vartheta_1(s),\mu(s))-f(s,\vartheta_2(s),\mu(s))\|^{2} \right)\\\\
\le& M^2\, \mathbf{K}\left( \displaystyle\int_{-\infty}^{t}e^{-\delta(t-s)}ds\right)^2.\; \displaystyle\sup_{t\in \mathbb{R}} \left(\Ex\|\vartheta_1(t)-\vartheta_2(t)\|^2\right)\\\\
\le& \dfrac{M^2\, \mathbf{K}}{\delta^2}\;\displaystyle\sup_{t\in \mathbb{R}} \left(\Ex\|\vartheta_1(t)-\vartheta_2(t)\|^2\right)\\\\
\end{array}
\end{equation}
From the hypotheses $\mathbf{(H_0)}$ and $\mathbf{(H_2)}$ and using Lemma \ref{Bridge},  we have
\begin{equation}\label{Estimation of P2}
\begin{array}{rl}
P_2=&\Ex\left\|\displaystyle\int_{-\infty}^{t} U(t,s)[\theta(s,\vartheta_1(s),\mu(s))-\theta(s,\vartheta_2(s),\mu(s))]dW(s)\right\|^2\\\\
\le& \widetilde{C_2} \left(\displaystyle\int_{-\infty}^{t} \|U(t,s)\|^2\, \Ex\|\theta(s,\vartheta_1(s),\mu(s))-\theta(s,\vartheta_2(s),\mu(s))\|^2_{\mathbb{L}_2^0}\;ds\right)\\\\
\le& \widetilde{C_2}{\mathbf{K}}M^2  \left(\displaystyle\int_{-\infty}^{t} e^{-2\delta(t-s)}  \Ex\|\vartheta_1(s)-\vartheta_2(s)\|^2ds\right)\\\\
\le& \widetilde{C_2} {\mathbf{K}} M^2 \left(\displaystyle\int_{-\infty}^{t} e^{-2\delta(t-s)} ds\right) \displaystyle \sup_{t\in \mathbb{R}}\left(\Ex\|\vartheta_1(t)-\vartheta_2(t)\|^2\right)\\\\
\le& \dfrac{\widetilde{C_2} {\mathbf{K}} M^2}{2\delta}\displaystyle \sup_{t\in \mathbb{R}}\left(\Ex\|\vartheta_1(t)-\vartheta_2(t)\|^2\right).
\end{array}
\end{equation}

Since $2\mathbf{K}M^2\left(\dfrac{1}{\delta^2}+\dfrac{\widetilde{C_2}}{2\delta}\right)<1$, we conclude that $\Upsilon$ is a contraction mapping in $\mathcal{C}_b(\mathbb{R},\mathcal{L}^2(\Proba,\Hilbert))$. Thus, $\Upsilon$ has a unique fixed point $\vartheta_\mu$.\\

\textbf{Step 2}. Next, we show that $\mu$ is the probability law of $\vartheta_\mu$.\\
Let $\mathcal{L}(\vartheta_\mu)=\{\mathcal{L}(\vartheta_\mu(t)),\, t\in \mathbb{R}\}$ represent the probability law of $\vartheta_\mu$ and define $\mathbf{Q}: \mathcal{C}_b(\mathbb{R},\Borel(\Hilbert))\to \mathcal{C}_b(\mathbb{R},\Borel(\Hilbert))$ by $\mathbf{Q}(\mu)=\mathcal{L}(\vartheta_\mu)$. We claim that  $\mathbf{Q}$ has a unique fixed-point. We use the Banach fixed-point-theorem and divide the proof in to two steps.\\
\textbf{Claim 1.} For arbitrary $\mu\in \mathcal{C}_b(\mathbb{R},\Borel(\Hilbert))$, the map $t\to \mathcal{L}(\vartheta_\mu(t))$ is continuous on $\mathbb{R}$.\\
To see this, let $ t\in \mathbb{R}$ and $|e|$ be sufficiently small.
We observe that
\begin{equation*}
\begin{array}{rl}
\mathcal{W}^2\,(\mathbf{Q}(\mu)(t+e),\mathbf{Q}(\mu)(t))=& \mathcal{W}^2\,(\mathbb{P}_{\vartheta_\mu(t+e)},\mathbb{P}_{\vartheta_\mu(t)})\\\\
\le & \Ex \| \vartheta_\mu(t+e)-\vartheta_\mu(t)\|^2\longrightarrow 0\quad\mbox{ as }\quad e\longrightarrow 0.
\end{array}
\end{equation*}
It follows that $t\to \mathbf{Q}(\mu)(t)$ is continuous on $\mathbb{R}$.\\
\textbf{Claim 2.} $\mathbf{Q}$ is a contraction mapping.\\
Let $\mu, \widetilde{\mu}\in \mathcal{C}_b(\mathbb{R},\Borel(\Hilbert))$ and $t\in \mathbb{R}$.
\begin{equation*}
\begin{array}{rl}
&\Ex\| \vartheta_{\mu}(t)-\vartheta_{\widetilde{\mu}}(t)\|^2 \\\\

\le& 3\Ex\big\|\displaystyle\int_{-\infty}^{t} U(t,s)\left[f(s,\vartheta_\mu(s),\mu(s))-f(s,\vartheta_{\widetilde{\mu}}(s),\widetilde{\mu}(s))\right]ds\big\|^2 \\\\
& 3\Ex\big\|\displaystyle\int_{-\infty}^{t} U(t,s)\left[\theta(s,\vartheta_\mu(s),\mu(s))-\theta(s,\vartheta_{\widetilde{\mu}}(s),\widetilde{\mu}(s))\right]dW(s)\big\|^2\\\\
& 3\Ex\big\|\displaystyle\int_{-\infty}^{t} U(t,s)\left[\psi(s,\mu(s))-\psi(s,\widetilde{\mu}(s))\right]dB^H(s)\big\|^2 \\\\
:=& \overline{J_1}+\overline{J_2}+\overline{J_3}
\end{array}
\end{equation*}

For $\overline{J_1}$, observe that
\begin{equation}\label{Mean1}
\begin{array}{rl}
\overline{J_1}=&3\Ex\big\|\displaystyle\int_{-\infty}^{t} U(t,s)\left[f(s,\vartheta_\mu(s),\mu(s))-f(s,\vartheta_{\widetilde{\mu}}(s),\widetilde{\mu}(s))\right]ds\big\|^2 \\\\
\le&3\Ex\left(\displaystyle\int_{-\infty}^{t} Me^{-\delta(t-s)}\|f(s,\vartheta_\mu(s),\mu(s))-f(s,\vartheta_{\widetilde{\mu}}(s),\widetilde{\mu}(s))\|ds\right)^2\\\\
\le&3\,M^2\,\left(\displaystyle\int_{-\infty}^{t} e^{-\delta(t-s)}ds\right)\,\Ex\left(\displaystyle\int_{-\infty}^{t} e^{-\delta(t-s)}\|f(s,\vartheta_\mu(s),\mu(s))-f(s,\vartheta_{\widetilde{\mu}}(s),\widetilde{\mu}(s))\|^2ds\right)\\\\
\le&3\,M^2\,\left(\displaystyle\int_{-\infty}^{t} e^{-\delta(t-s)}ds\right)\,\left(\displaystyle\int_{-\infty}^{t} e^{-\delta(t-s)}\mathbf{K}\, \left(\Ex\|\vartheta_\mu(s)-\vartheta_{\widetilde{\mu}}(s)\|^2+\mathcal{W}^2\left(\mu(s);\widetilde{\mu}(s)\right) ds\right)\right)\\\\

\le&\dfrac{3\,M^2\,\mathbf{K}}{\delta}\,\left(\displaystyle\int_{-\infty}^{t} e^{-\delta(t-s)}\, \left(\Ex\|\vartheta_\mu(s)-\vartheta_{\widetilde{\mu}}(s)\|^2+\mathcal{W}^2\left(\mu(s);\widetilde{\mu}(s)\right) ds\right)\right)\\\\
\le&
\dfrac{3\,M^2\,\mathbf{K}}{\delta}\,\displaystyle\int_{-\infty}^{t} e^{-\delta(t-s)}\, \Ex\|\vartheta_\mu(s)-\vartheta_{\widetilde{\mu}}(s)\|^2ds

+\dfrac{3\,M^2\,\mathbf{K}}{\delta}\,\left(\displaystyle\int_{-\infty}^{t} e^{-\delta(t-s)}\,\mathcal{W}^2\left(\mu(s);\widetilde{\mu}(s)\right) ds\right)\\\\
\end{array}
\end{equation}

For $\overline{J_2}$, we have

\begin{equation}\label{Mean2}
\begin{array}{rl}
\overline{J_2}=&3\Ex\big\|\displaystyle\int_{-\infty}^{t} U(t,s)\left[\theta(s,\vartheta_\mu(s),\mu(s))-\theta(s,\vartheta_{\widetilde{\mu}}(s),\widetilde{\mu}(s))\right]dW(s)\big\|^2\\\\

\le& \widetilde{C_2} \Ex\left(\displaystyle\int_{-\infty}^{t} \|U(t,s)\|^2\|\theta(s,\vartheta_\mu(s),{\mu}(s))-\theta(s,\vartheta_{\widetilde{\mu}}(s),{\widetilde{\mu}(s)})\|^2_{\mathbb{L}_2^0}\;ds\right)\\\\

\le& \widetilde{C_2}  \left(\displaystyle\int_{-\infty}^{t} M^2 e^{-\delta(t-s)} ds\right)  \left(\displaystyle\int_{-\infty}^{t} e^{-\delta(t-s)} {\mathbf{K}} \left(\Ex\|\vartheta_\mu(s)-\vartheta_{\widetilde{\mu}}(s)\|^2+\mathcal{W}^2({\mu}(s),{\widetilde{\mu}(s)}\right)ds\right)\\\\

\le& \dfrac{\widetilde{C_2}\, M^2\,{\mathbf{K}}}{\delta}  \left(\displaystyle\int_{-\infty}^{t} e^{-\delta(t-s)}  \left(\Ex\|\vartheta_\mu(s)-\vartheta_{\widetilde{\mu}}(s)\|^2+\mathcal{W}^2({\mu}(s),{\widetilde{\mu}(s)}\right)ds\right)\\\\

\le& \dfrac{\widetilde{C_2}\, M^2\,{\mathbf{K}}}{\delta}  \displaystyle\int_{-\infty}^{t} e^{-\delta(t-s)}  \Ex\|\vartheta_\mu(s)-\vartheta_{\widetilde{\mu}}(s)\|^2\,ds+\dfrac{\widetilde{C_2}\, M^2\,{\mathbf{K}}}{\delta} \displaystyle\int_{-\infty}^{t} e^{-\delta(t-s)} \mathcal{W}^2({\mu}(s),{\widetilde{\mu}(s)})\,ds\\\\
\end{array}
\end{equation}

For $\overline{J_3}$, we have
\begin{equation*}
\begin{array}{rl}
\overline{J_3}=&3\Ex\big\|\displaystyle\int_{-\infty}^{t} U(t,s)\left[\psi(s,\mu(s))-\psi(s,\widetilde{\mu}(s))\right]dB^H(s)\big\|^2 \\

\le& 3H(2H-1)\displaystyle\int_{-\infty}^{t}\int_{-\infty}^{t}\|U(t,s)[\psi(s,\mu(s))-\psi(s,{\widetilde{\mu}(s)})]\|_{\mathbb{L}_2} \\\\
&\qquad\qquad\qquad\times
\|U(t,r)[\psi(r,{\mu}(r))-\psi(r,{\widetilde{\mu}(r)})]\|_{\mathbb{L}_2}|r-s|^{2H-2}dr\,ds\\\\
\le& 3H(2H-1)M^2 (\mathbf{K})^2\displaystyle\int_{-\infty}^{t}\int_{-\infty}^{t}e^{-\delta(t-s)}\mathcal{W}(\mu(s),{\widetilde{\mu}(s)}) \\\\
&\qquad\qquad\qquad\times
e^{-\delta(t-r)}\mathcal{W}({\mu}(r),{\widetilde{\mu}(r)})|r-s|^{2H-2}dr\,ds\\\\
\le&3 H(2H-1)M^2 (\mathbf{K})^2\displaystyle\int_{0}^{+\infty}\int_{0}^{+\infty}e^{-\delta\,y}\,\mathcal{W}({\mu}(t-y),{\widetilde{\mu}(t-y)}) \\\\
&\qquad\qquad\qquad\times
e^{-\delta\,x}\,\mathcal{W}({\mu}(t-x),{\widetilde{\mu}(t-x)})|x-y|^{2H-2}dx\,dy\\\\

\end{array}
\end{equation*}

\begin{equation}\label{Mean3}
\begin{array}{rl}

\le&3 H(2H-1)M^2 (\mathbf{K})^2\displaystyle\int_{0}^{+\infty}e^{-\delta\,y}\,\mathcal{W}({\mu}(t-y),{\widetilde{\mu}(t-y)}) \\\\
&\qquad\times
\displaystyle\left(\int_{0}^{+\infty}e^{-\delta\,x}\,\mathcal{W}^2({\mu}(t-x),{\widetilde{\mu}(t-x)})dx\right)^{1/2}\left(\int_{0}^{+\infty}e^{-\delta\,x}|x-y|^{4H-4}dx\right)^{1/2}\,dy\\\\
\le& 3H(2H-1)M^2 (\mathbf{K})^2\left(\displaystyle\int_{0}^{+\infty}e^{-\delta\,y}\,\mathcal{W}^2({\mu}(t-y),{\widetilde{\mu}(t-y)}) dy\right)^{1/2}\\\\
&\times
\displaystyle\left(\int_{0}^{+\infty}e^{-\delta\,x}\,\mathcal{W}^2({\mu}(t-x),{\widetilde{\mu}(t-x)})dx\right)^{1/2}\left(\displaystyle\int_{0}^{+\infty}e^{-\delta\,y}\,\int_{0}^{+\infty}e^{-\delta\,x}|x-y|^{4H-4}dx\,dy\right)^{1/2}\\\\

\le& 3H(2H-1)M^2 (\mathbf{K})^2
\left(\dfrac{\Gamma(4H-2)}{(4H-2)\delta^{4H-2}}+\dfrac{\Gamma(4H-3)}{2}\right)^{1/2}\left(\displaystyle\int_{-\infty}^{t}e^{-\delta\,(t-s)}\,\mathcal{W}^2({\mu}(s),{\widetilde{\mu}(s)}) ds\right)\\\\
\le&3\dfrac{H(2H-1)M^2 (\mathbf{K})^2}{\delta^{2H-1}}\left(\dfrac{\Gamma(4H-2)}{4H-2}+\dfrac{\Gamma(4H-3)}{2}\right)^{1/2}\left(\displaystyle\int_{-\infty}^{t}e^{-\delta\,(t-s)}\,\mathcal{W}^2({\mu}(s),{\widetilde{\mu}(s)}) ds\right)\\\\
\le&3{H(2H-1)M^2 (\mathbf{K})^2}{\delta}\,C(\delta,H)\left(\displaystyle\int_{-\infty}^{t}e^{-\delta\,(t-s)}\,\mathcal{W}^2({\mu}(s),{\widetilde{\mu}(s)}) ds\right)\\\\
\end{array}
\end{equation}

Hence, from \eqref{Mean1}, \eqref{Mean2} and \eqref{Mean3}, we obtain
\begin{equation*}
\begin{array}{rl}
&\Ex\| \vartheta_{\mu}(t)-\vartheta_{\widetilde{\mu}}(t)\|^2 \\\\
\le&
\dfrac{3\,M^2\,\mathbf{K}}{\delta}\,\displaystyle\int_{-\infty}^{t} e^{-\delta(t-s)}\, \Ex\|\vartheta_\mu(s)-\vartheta_{\widetilde{\mu}}(s)\|^2ds

+\dfrac{3\,M^2\,\mathbf{K}}{\delta}\,\left(\displaystyle\int_{-\infty}^{t} e^{-\delta(t-s)}\,\mathcal{W}^2\left(\mu(s);\widetilde{\mu}(s)\right) ds\right)\\\\
&+ \dfrac{\widetilde{C_2}\, M^2\,{\mathbf{K}}}{\delta}  \displaystyle\int_{-\infty}^{t} e^{-\delta(t-s)}  \Ex\|\vartheta_\mu(s)-\vartheta_{\widetilde{\mu}}(s)\|^2\,ds+\dfrac{\widetilde{C_2}\, M^2\,{\mathbf{K}}}{\delta} \displaystyle\int_{-\infty}^{t} e^{-\delta(t-s)} \mathcal{W}^2({\mu}(s),{\widetilde{\mu}(s)})\,ds\\\\
&+{3H(2H-1)M^2 (\mathbf{K})^2}{\delta}\,C(\delta,H)\left(\displaystyle\int_{-\infty}^{t}e^{-\delta\,(t-s)}\,\mathcal{W}^2({\mu}(s),{\widetilde{\mu}(s)}) ds\right)\\\\

\end{array}
\end{equation*}

\begin{equation*}
\begin{array}{rl}
\le& \dfrac{3\,M^2\,\mathbf{K}+\widetilde{C_2}\, M^2\,{\mathbf{K}}}{\delta}\,\displaystyle\int_{-\infty}^{t} e^{-\delta(t-s)}\, \Ex\|\vartheta_\mu(s)-\vartheta_{\widetilde{\mu}}(s)\|^2ds\\\\
&\,+\left( \dfrac{3\,M^2\,\mathbf{K}+\widetilde{C_2}\, M^2\,{\mathbf{K}}}{\delta}+{3H(2H-1)M^2 (\mathbf{K})^2}{\delta}\,C(\delta,H)\right)\displaystyle\int_{-\infty}^{t}e^{-\delta\,(t-s)}\,\mathcal{W}^2({\mu}(s),{\widetilde{\mu}(s)}) ds\\\\
\le& \beta_1\,\displaystyle\int_{-\infty}^{t} e^{-\delta(t-s)}\, \Ex\|\vartheta_\mu(s)-\vartheta_{\widetilde{\mu}}(s)\|^2ds+\beta_2\displaystyle\int_{-\infty}^{t}e^{-\delta\,(t-s)}\,\mathcal{W}^2({\mu}(s),{\widetilde{\mu}(s)}) ds\\\\
\le& \beta_1\,\displaystyle\int_{-\infty}^{t} e^{-\delta(t-s)}\, \Ex\|\vartheta_\mu(s)-\vartheta_{\widetilde{\mu}}(s)\|^2ds+\left(\dfrac{\beta_2}{\delta}\right)\displaystyle\sup_{t\in \mathbb{R}}\mathcal{W}^2({\mu}(t),{\widetilde{\mu}(t)}),\\\\
\end{array}
\end{equation*}
where \begin{equation}\label{beta12}
\beta_1=\dfrac{3\,M^2\,\mathbf{K}+\widetilde{C_2}\, M^2\,{\mathbf{K}}}{\delta} \;\mbox{ and }
\beta_2= \beta_1+{3H(2H-1)M^2 (\mathbf{K})^2}{\delta}\,C(\delta,H).
\end{equation}

An application of Lemma \ref{Gronw}, yields
\begin{equation*}
\begin{array}{rl}
&\Ex\| \vartheta_{\mu}(t)-\vartheta_{\widetilde{\mu}}(t)\|^2 \\\\
\le&\left(\dfrac{\beta_2}{\delta}\right)\displaystyle\sup_{t\in \mathbb{R}}\mathcal{W}^2({\mu}(t),{\widetilde{\mu}(t)})+\beta_1\,\displaystyle\int_{-\infty}^{t} e^{-\delta(t-s)}\,\left(\dfrac{\beta_2}{\delta}\right)\displaystyle\sup_{t\in \mathbb{R}}\mathcal{W}^2({\mu}(t),{\widetilde{\mu}(t)})ds\\\\
\le&\left(\dfrac{\beta_2}{\delta}\right)\displaystyle\sup_{t\in \mathbb{R}}\mathcal{W}^2({\mu}(t),{\widetilde{\mu}(t)})+\beta_2\,\left(\dfrac{\beta_2}{\delta^2}\right)\displaystyle\sup_{t\in \mathbb{R}}\mathcal{W}^2({\mu}(t),{\widetilde{\mu}(t)})\\\\
\end{array}
\end{equation*}
\begin{equation*}
\begin{array}{rl}
\le&\left[\dfrac{\beta_2}{\delta}+\,\dfrac{\beta_2\beta_2}{\delta^2}\right]\displaystyle\sup_{t\in \mathbb{R}}\mathcal{W}^2({\mu}(t),{\widetilde{\mu}(t)})\\\\
\le&\dfrac{\beta_2}{\delta}\left[1+\,\dfrac{\beta_2}{\delta}\right]\displaystyle\sup_{t\in \mathbb{R}}\mathcal{W}^2({\mu}(t),{\widetilde{\mu}(t)})

\end{array}
\end{equation*}
Therefore,

\begin{equation*}
\displaystyle\sup_{t\in \mathbb{R}}\mathcal{W}^2(\mathbf{Q}(\mu)(t),\mathbf{Q}(\widetilde{\mu})(t))\le\dfrac{\beta_2}{\delta}\left[1+\,\dfrac{\beta_2}{\delta}\right]\displaystyle\sup_{t\in \mathbb{R}}\mathcal{W}^2({\mu}(t),{\widetilde{\mu}(t)}).
\end{equation*}
Since $\dfrac{\beta_2}{\delta}\left[1+\,\dfrac{\beta_2}{\delta}\right]<1$, it follows that $\mathbf{Q}$ is a contraction mapping on $\mathcal{C}_b(\mathbb{R},\Borel(\Hilbert))$. Therefore, by the Banach fixed-point theorem, we deduce that $\mathbf{Q}$ has a unique fixed-point $\mu$ and $\vartheta_{\mu}$ is a mild solution of Eq.\eqref{C1} on $\mathbb{R}$.\\

\noindent\textbf{Step 3 :} We show the almost automorphic in distribution of $\mathcal{L}^2$-bounded solution for Equ.\eqref{C1}.\\
Let $\{e^\prime_n\}$ be an arbitrary sequence of real numbers. Since $f,\psi$ and $\theta$ are square-mean almost automorphic, there exists a subsequence $\{e_n\}$ of $\{e^\prime_n\}$ and  functions $\widehat{f},\widehat{\psi}$ and $\widehat{\theta}$ such that
\begin{equation*}
\begin{array}{rl}
\displaystyle\lim\limits_{n\to \infty}\Ex\|f(t+e_n,\vartheta,\Proba_{\vartheta})-\widehat{f}(t,\vartheta,\Proba_{\vartheta})\|^2&=0 \mbox{ and }
\displaystyle\lim\limits_{n\to \infty}\Ex\|\widehat{f}(t-e_n,\vartheta,\Proba_{\vartheta})-f(t,\vartheta,\Proba_{\vartheta})\|^2=0,\\\\
\displaystyle\lim\limits_{n\to
\infty}\Ex\|\theta(t+e_n,\vartheta,\Proba_{\vartheta})-\widehat{\theta}(t,\vartheta,\Proba_{\vartheta})\|^2_{\mathbb{L}_2^0}&=0\mbox{ and }
\displaystyle\lim\limits_{n\to \infty}\Ex\|\widehat{\theta}(t-e_n,\vartheta,\Proba_{\vartheta})-\theta(t,\vartheta,\Proba_{\vartheta})\|^2_{\mathbb{L}_2^0}=0\\\\

\displaystyle\lim\limits_{n\to
	\infty}\Ex\|\psi(t+e_n,\Proba_{\vartheta})-\widehat{\psi}(t,\Proba_{\vartheta})\|^2_{\mathbb{L}_2}&=0\mbox{ and }
\displaystyle\lim\limits_{n\to \infty}\Ex\|\widehat{\psi}(t-e_n,\Proba_{\vartheta})-\psi(t,\Proba_{\vartheta})\|^2_{\mathbb{L}_2}=0\\\\
\end{array}
\end{equation*}
for each $t\in \mathbb{R}$, $\vartheta\in \mathcal{L}^2(\Proba,\Hilbert),$ and $\Proba_\vartheta\in \Borel(\Hilbert)$. By $\mathbf{(H_{0})}$, there exist an evolution family $V(t,s)$ and a bounded subset $B$ of $\mathcal{L}^2(\Proba, \Hilbert)$ such that
\begin{equation}\label{Evo1}
\lim\limits_{n\to \infty} \Ex\| U(t+e_n,s+e_n)\vartheta -V(t,s)\vartheta\|^2=0
\end{equation}
and
\begin{equation}
\lim\limits_{n\to \infty} \Ex\| V(t-e_n,s-e_n)\vartheta -U(t,s)\vartheta\|^2=0,
\end{equation} for each $\vartheta\in B$. By \eqref{Evo1} and the exponential dissipation property of $U(t,s)$, we have
\begin{equation}\label{EVF}
	\Ex\|V(t,s)\vartheta\|^2\le 2M^2\, e^{-2\delta(t-s)}\Ex\|\vartheta\|^2 \mbox{ for all } t\ge s \mbox{ and } \vartheta\in B.
\end{equation}
   Let
$\widehat{\vartheta}(\cdot)$ be such that
\begin{equation}
\begin{array}{rl}
\widehat{\vartheta}(t)=&\displaystyle\int_{-\infty}^{t} V(t,s)\widehat{f}(s,\widehat{\vartheta}(s),\Proba_{\widehat{\vartheta}(s)})ds + \displaystyle\int_{-\infty}^{t} V(t,s)\widehat{\theta}(s,\widehat{\vartheta}(s),\Proba_{\widehat{\vartheta}(s)})dW(s)\\\\
& + \displaystyle\int_{-\infty}^{t} V(t,s)\widehat{\psi}(s,\Proba_{\widehat{\vartheta}(s)})dB^H(s),
\end{array}
\end{equation} and for each $s\in \mathbb{R}$ let
 $\widehat{W}_n(s)=W(s+e_n)-W(e_n)$ and $\widehat{B^H_n}(s)=B^H(s+e_n)-B^H(e_n)$. We know that $\widehat{W}_n$ is a Brownian motion with same law as $W$ and $\widehat{B^H}$ is a fractional Brownian motion with the same law as $B^H$.
 The process
\begin{equation}
\begin{array}{ll}
\vartheta(t+e_n)=&\displaystyle\int_{-\infty}^{t+e_n} U(t+e_n,s)f(s,\vartheta(s),\Proba_{\vartheta(s)})ds + \displaystyle\int_{-\infty}^{t+e_n} U(t+e_n,s)\theta(s,\vartheta(s),\Proba_{\vartheta(s)})dW(s)\\\\
& + \displaystyle\int_{-\infty}^{t+e_n} U(t+e_n,s)\psi(s,\Proba_{\vartheta(s)})dB^H(s)
\end{array}
\end{equation}

becomes
\begin{equation}
\begin{array}{ll}
\vartheta(t+e_n)=&\displaystyle\int_{-\infty}^{t} U(t+e_n,s+e_n)f(s+e_n,\vartheta(s+e_n),\Proba_{\vartheta(s+e_n)})ds \\\\
&\qquad+ \displaystyle\int_{-\infty}^{t} U(t+e_n,s+e_n)\theta(s+e_n,\vartheta(s+e_n),\Proba_{\vartheta(s+e_n)})d\widehat{\mathbb{W_n}}(s)\\\\
& \qquad\qquad+ \displaystyle\int_{-\infty}^{t } U(t+e_n ,s+e_n)\psi(s+e_n,\Proba_{\vartheta(s+e_n)})d\widehat{B^H_n}(s).
\end{array}
\end{equation}
We consider the process
\begin{equation}
\begin{array}{ll}
\vartheta_n(t)=&\displaystyle\int_{-\infty}^{t} U(t+e_n,s+e_n)f(s+e_n,\vartheta_n(s),\Proba_{\vartheta_n(s)})ds \\\\
&\qquad+ \displaystyle\int_{-\infty}^{t} U(t,s)\theta(s+e_n,\vartheta_n(s),\Proba_{\vartheta_n(s)})d{W}(s)\\\\
& \qquad\qquad+ \displaystyle\int_{-\infty}^{t } U(t ,s)\psi(s+e_n,\Proba_{\vartheta_n(s)})d{B^H}(s).
\end{array}
\end{equation}
Note that $\vartheta(t+e_n)$  has the same distribution as $\vartheta_n(t)$ for each $t\in \mathbb{R}$.
We claim that $\vartheta_n(t)$ converges in quadratic mean to $\widehat{\vartheta}(t)$ for each fixed $t\in \mathbb{R}$. To see this, observe that
\begin{equation}
\begin{array}{rl}
&\Ex\| \vartheta_n(t)-\widehat{\vartheta}(t)\|^2\\\\
\le &3\Ex\| \displaystyle\int_{-\infty}^{t} \left[U(t+e_n,s+e_n)f\left(s+e_n,\vartheta_n(s),\Proba_{\vartheta_n(s)}\right)- V(t,s)\widehat{f}\left(s,\widehat{\vartheta}(s),\Proba_{\widehat{\vartheta}(s)}\right)\right]ds\|^2 \\\\
&+3\Ex\|\displaystyle\int_{-\infty}^{t} \left[U(t+e_n,s+e_n)\theta\left(s+e_n,\vartheta_n(s),\Proba_{\vartheta_n(s)}\right) -V(t,s)\widehat{\theta}\left(s,\widehat{\vartheta}(s),\Proba_{\widehat{\vartheta}(s)}\right)\right]dW(s)\|^2\\\\
&+3\Ex\| \displaystyle\int_{-\infty}^{t } \left[U(t+e_n,s+e_n)\psi\left(s+e_n,\Proba_{\vartheta_n(s)}\right)-V(t ,s)\widehat{\psi}\left(s,\Proba_{\widehat{\vartheta}(s)}\right)\right]dB^H(s)\|^2\\\\
&:=J_1+J_2+J_3

\end{array}
\end{equation}

For $J_1$, we have

\begin{equation*}
\begin{array}{rl}
J_1\le& 3\Ex\| \displaystyle\int_{-\infty}^{t} \left[U(t+e_n,s+e_n)f\left(s+e_n,\vartheta_n(s),\Proba_{\vartheta_n(s)}\right)- V(t,s)\widehat{f}\left(s,\widehat{\vartheta}(s),\Proba_{\widehat{\vartheta}(s)}\right)\right]ds\|^2 \\\\
\le& 9\Ex\| \displaystyle\int_{-\infty}^{t}U(t+e_n,s+e_n)\left[f\left(s+e_n,\vartheta_n(s),\Proba_{\vartheta_n(s)}\right)-f\left(s+e_n,\widehat{\vartheta}(s),\Proba_{\widehat{\vartheta}(s)}\right) \right]ds\|^2\\\\
&+9\Ex\| \displaystyle\int_{-\infty}^{t} U(t+e_n,s+e_n)\left[f\left(s+e_n,\widehat{\vartheta}(s),\Proba_{\widehat{\vartheta}(s)}\right)-\widehat{f}\left(s,\widehat{\vartheta}(s),\Proba_{\widehat{\vartheta}(s)}\right) \right]ds\|^2\\\\
&+9\Ex\| \displaystyle\int_{-\infty}^{t}\left[ U(t+e_n,s+e_n)\widehat{f}\left(s,\widehat{\vartheta}(s),\Proba_{\widehat{\vartheta}(s)}\right) - V(t,s)\widehat{f}\left(s,\widehat{\vartheta}(s),\Proba_{\widehat{\vartheta}(s)}\right)\right]ds\|^2\\\\
\le& 9M^2\Ex\bigg(\displaystyle\int_{-\infty}^{t}e^{-\delta(t-s)}\; \|f\left(s+e_n,\vartheta_n(s),\Proba_{\vartheta_n(s)}\right)-f\left(s+e_n,\widehat{\vartheta}(s),\Proba_{\widehat{\vartheta}(s)}\right) \|\,ds\bigg)^2\\\\
&+9M^2\Ex\bigg( \displaystyle\int_{-\infty}^{t} e^{-\delta(t-s)}\;\|f\left(s+e_n,\widehat{\vartheta}(s),\Proba_{\widehat{\vartheta}(s)}\right)-\widehat{f}\left(s,\widehat{\vartheta}(s),\Proba_{\widehat{\vartheta}(s)}\right) \|\,ds\bigg)^2\\\\
&+9\Ex\bigg(\displaystyle\int_{-\infty}^{t}\|  U(t+e_n,s+e_n)\widehat{f}\left(s,\widehat{\vartheta}(s),\Proba_{\widehat{\vartheta}(s)}\right) - V(t,s)\widehat{f}\left(s,\widehat{\vartheta}(s),\Proba_{\widehat{\vartheta}(s)}\right)\|\,ds\bigg)^2\\\\
\le& 9\displaystyle\frac{M^2}{\delta}\displaystyle\int_{-\infty}^{t}e^{-\delta(t-s)}\; \Ex\|f\left(s+e_n,\vartheta_n(s),\Proba_{\vartheta_n(s)}\right)-f\left(s+e_n,\widehat{\vartheta}(s),\Proba_{\widehat{\vartheta}(s)}\right) \|^2\,ds\\\\
&+9\displaystyle\frac{M^2}{\delta}\displaystyle\int_{-\infty}^{t} e^{-\delta(t-s)}\;\Ex\|f\left(s+e_n,\widehat{\vartheta}(s),\Proba_{\widehat{\vartheta}(s)}\right)-\widehat{f}\left(s,\widehat{\vartheta}(s),\Proba_{\widehat{\vartheta}(s)}\right) \|^2\,ds\\\\
&+9\Ex\bigg(\displaystyle\int_{-\infty}^{t}\|  U(t+e_n,s+e_n)\widehat{f}\left(s,\widehat{\vartheta}(s),\Proba_{\widehat{\vartheta}(s)}\right) - V(t,s)\widehat{f}\left(s,\widehat{\vartheta}(s),\Proba_{\widehat{\vartheta}(s)}\right)\|\,ds\bigg)^2\\\\
\le& 9\displaystyle\frac{M^2}{\delta}\displaystyle\int_{-\infty}^{t}e^{-\delta(t-s)}\; \left(\Ex\left\| \vartheta_n(s)-\widehat{\vartheta}(s)\right\|^2+\mathcal{W}^2(\Proba_{\vartheta_n(s)},\Proba_{\widehat{\vartheta}(s)}\right)\,ds\\\\
&+9\displaystyle\frac{M^2}{\delta}\displaystyle\int_{-\infty}^{t} e^{-\delta(t-s)}\;\Ex\|f\left(s+e_n,\widehat{\vartheta}(s),\Proba_{\widehat{\vartheta}(s)}\right)-\widehat{f}\left(s,\widehat{\vartheta}(s),\Proba_{\widehat{\vartheta}(s)}\right) \|^2\,ds\\\\
&+9\bigg(\displaystyle\int_{-\infty}^{t}e^{-p(t-s)}ds\bigg) \bigg(\displaystyle\int_{-\infty}^{t}e^{p(t-s)}\Ex\| [U(t+e_n,s+e_n) - V(t,s)]\widehat{f}\left(s,\widehat{\vartheta}(s),\Proba_{\widehat{\vartheta}(s)}\right)\|^2\,ds\bigg)\\\\
\end{array}
\end{equation*}

\begin{equation}
\begin{array}{rl}
\le& 18\displaystyle\frac{M^2}{\delta}\displaystyle\int_{-\infty}^{t}e^{-\delta(t-s)}\; \Ex\left\|\vartheta_n(s)-\widehat{\vartheta}(s)\right\|^2 ds\\\\
&+9\displaystyle\frac{M^2}{\delta}\displaystyle\int_{-\infty}^{t} e^{-\delta(t-s)}\;\Ex\|f\left(s+e_n,\widehat{\vartheta}(s),\Proba_{\widehat{\vartheta}(s)}\right)-\widehat{f}\left(s,\widehat{\vartheta}(s),\Proba_{\widehat{\vartheta}(s)}\right) \|^2\,ds\\\\
&+\displaystyle\frac{9}{p}  \bigg(\displaystyle\int_{-\infty}^{t}e^{p(t-s)}\Ex\| [U(t+e_n,s+e_n) - V(t,s)]\widehat{f}\left(s,\widehat{\vartheta}(s),\Proba_{\widehat{\vartheta}(s)}\right)\|^2\,ds\bigg)\\\\
\le& \dfrac{18\mathbf{K}\,M^2}{\delta}\, \displaystyle\int_{-\infty}^{t}\,e^{-\delta(t-s)}\; \Ex\|\vartheta_n(s)-\widehat{\vartheta}(s)\|^2\,ds + X_{1}({n})
\end{array}
\end{equation}

where $p\in (0,2\delta)$ is some constant and
\begin{equation*}
\begin{array}{rl}
X_{1}({n})=&9\displaystyle\frac{M^2}{\delta}\displaystyle\int_{-\infty}^{t} e^{-\delta(t-s)}\;\Ex\|f\left(s+e_n,\widehat{\vartheta}(s),\Proba_{\widehat{\vartheta}(s)}\right)-\widehat{f}\left(s,\widehat{\vartheta}(s),\Proba_{\widehat{\vartheta}(s)}\right) \|^2\,ds\\\\
&+\displaystyle\frac{9}{p}  \bigg(\displaystyle\int_{-\infty}^{t}e^{p(t-s)}\Ex\| [U(t+e_n,s+e_n) - V(t,s)]\widehat{f}\left(s,\widehat{\vartheta}(s),\Proba_{\widehat{\vartheta}(s)}\right)\|^2\,ds\bigg).\\\\
\end{array}
\end{equation*}

Since  $f$ is square-mean almost automorphic in $t$ and $\widehat{\vartheta}(\cdot)$ is bounded in $\mathcal{L}^2(\Proba, \Hilbert)$, we have $\sup_{s\in \mathbb{R}}\|f\left(s+e_n,\widehat{\vartheta}(s),\Proba_{\widehat{\vartheta}(s)}\right)\|^2<\infty$, so that $\sup_{s\in \mathbb{R}}\|\widehat{f}\left(s,\widehat{\vartheta}(s),\Proba_{\widehat{\vartheta}(s)}\right)\|^2<\infty$. Noting \eqref{Evo1}, from Lebesgue dominated convergence theorem, it follows that

\begin{equation*}
	\begin{array}{rr}
	\lim\limits_{n\to \infty} \displaystyle \int_{-\infty}^{t} e^{-\delta(t-s)} \Ex\|f\left(s+e_n,\widehat{\vartheta}(s),\Proba_{\widehat{\vartheta}(s)}\right)-\widehat{f}\left(s,\widehat{\vartheta}(s),\Proba_{\widehat{\vartheta}(s)}\right)\|^2\,ds=&0\\\\
\lim\limits_{n\to \infty}\displaystyle\int_{-\infty}^{t}e^{p(t-s)}\Ex\| [U(t+e_n,s+e_n) - V(t,s)]\widehat{f}\left(s,\widehat{\vartheta}(s),\Proba_{\widehat{\vartheta}(s)}\right)\|^2\,ds=0\\
	\end{array}
\end{equation*}
Therefore, $X_1(n)\to 0$.\\
For $J_2$, by Lemma \ref{Bridge}, we have
\begin{equation*}
\begin{array}{rl}
J_2
\le& 9\Ex\| \displaystyle\int_{-\infty}^{t}U(t+e_n,s+e_n)\left[\theta\left(s+e_n,\vartheta_n(s),\Proba_{\vartheta_n(s)}\right)-\theta\left(s+e_n,\widehat{\vartheta}(s),\Proba_{\widehat{\vartheta}(s)}\right) \right]dW(s)\|^2\\\\
&+9\Ex\| \displaystyle\int_{-\infty}^{t} U(t+e_n,s+e_n)\left[\theta\left(s+e_n,\widehat{\vartheta}(s),\Proba_{\widehat{\vartheta}(s)}\right)-\widehat{\theta}\left(s,\widehat{\vartheta}(s),\Proba_{\widehat{\vartheta}(s)}\right) \right]dW(s)\|^2\\\\
&+9\Ex\| \displaystyle\int_{-\infty}^{t}\left[ U(t+e_n,s+e_n)\widehat{\theta}\left(s,\widehat{\vartheta}(s),\Proba_{\widehat{\vartheta}(s)}\right) - V(t,s)\widehat{\theta}\left(s,\widehat{\vartheta}(s),\Proba_{\widehat{\vartheta}(s)}\right)\right]dW(s)\|^2\\\\
	\end{array}
\end{equation*}
\begin{equation}\label{J2}
\begin{array}{rl}

\le& 9\widetilde{C_2} \displaystyle\int_{-\infty}^{t}\Ex\left\| U(t+e_n,s+e_n)\left[\theta\left(s+e_n,\vartheta_n(s),\Proba_{\vartheta_n(s)}\right)-\theta\left(s+e_n,\widehat{\vartheta}(s),\Proba_{\widehat{\vartheta}(s)}\right)\right]\right\|^2_{L_2^0}ds\\\\
&+9\widetilde{C_2}  \displaystyle\int_{-\infty}^{t}\Ex\left\| U(t+e_n,s+e_n)\left[\theta\left(s+e_n,\widehat{\vartheta}(s),\Proba_{\widehat{\vartheta}(s)}\right)-\widehat{\theta}\left(s,\widehat{\vartheta}(s),\Proba_{\widehat{\vartheta}(s)}\right) \right]\right\|^2_{L_2^0}ds\\\\
&+9\widetilde{C_2} \displaystyle\int_{-\infty}^{t}\Ex\left\| U(t+e_n,s+e_n)\widehat{\theta}\left(s,\widehat{\vartheta}(s),\Proba_{\widehat{\vartheta}(s)}\right) - V(t,s)\widehat{\theta}\left(s,\widehat{\vartheta}(s),\Proba_{\widehat{\vartheta}(s)}\right)\right\|^2_{L_2^0}ds\\\\
\le& 9\,\widetilde{C_2}\,M^2\, \mathbf{K} \displaystyle\int_{-\infty}^{t}  e^{-2\delta(t-s)}\left(\Ex\left\| \vartheta_n(s)-\widehat{\vartheta}(s)\right\|^2-\mathcal{W}^2(\Proba_{\vartheta_n(s)},\Proba_{\widehat{\vartheta}(s)}\right)\,ds\\\\
&+9\widetilde{C_2}\,M^2\,\displaystyle\int_{-\infty}^{t}e^{-2\delta(t-s)}\Ex\left\| \theta\left(s+e_n,\widehat{\vartheta}(s),\Proba_{\widehat{\vartheta}(s)}\right)-\widehat{\theta}\left(s,\widehat{\vartheta}(s),\Proba_{\widehat{\vartheta}(s)}\right) \right\|^2_{L_2^0}ds\\\\
&+9\widetilde{C_2} \displaystyle\int_{-\infty}^{t}\Ex\left\| U(t+e_n,s+e_n)\widehat{\theta}\left(s,\widehat{\vartheta}(s),\Proba_{\widehat{\vartheta}(s)}\right) - V(t,s)\widehat{\theta}\left(s,\widehat{\vartheta}(s),\Proba_{\widehat{\vartheta}(s)}\right)\right\|^2_{L_2^0}ds\\\\
\le& 18\,\widetilde{C_2}\,M^2\, \mathbf{K} \displaystyle\int_{-\infty}^{t}  e^{-2\delta(t-s)}\Ex\left\| \vartheta_n(s)-\widehat{\vartheta}(s)\right\|^2\,ds + X_2(n),
\end{array}
\end{equation}
where
\begin{equation*}
\begin{array}{ll}
X_2(n)=&9\widetilde{C_2}\,M^2\,\displaystyle\int_{-\infty}^{t}e^{-2\delta(t-s)}\Ex\left\| \theta\left(s+e_n,\widehat{\vartheta}(s),\Proba_{\widehat{\vartheta}(s)}\right)-\widehat{\theta}\left(s,\widehat{\vartheta}(s),\Proba_{\widehat{\vartheta}(s)}\right) \right\|^2_{L_2^0}ds\\\\
&+9\widetilde{C_2} \displaystyle\int_{-\infty}^{t}\Ex\left\| U(t+e_n,s+e_n)\widehat{\theta}\left(s,\widehat{\vartheta}(s),\Proba_{\widehat{\vartheta}(s)}\right) - V(t,s)\widehat{\theta}\left(s,\widehat{\vartheta}(s),\Proba_{\widehat{\vartheta}(s)}\right)\right\|^2_{L_2^0}ds.
\end{array}
\end{equation*}
Similarly, we have $\lim\limits_{n\to \infty}X_2(n)=0$.

%

\begin{equation*}
\begin{array}{rl}
J_3\le &3\Ex\| \displaystyle\int_{-\infty}^{t } \left[U(t+e_n,s+e_n)\psi\left(s+e_n,\Proba_{\vartheta_n(s)}\right)-V(t ,s)\widehat{\psi}\left(s,\Proba_{\widehat{\vartheta}(s)}\right)\right]dB^H(s)\|^2\\\\
\le& 9\Ex\| \displaystyle\int_{-\infty}^{t}U(t+e_n,s+e_n)\left[\psi\left(s+e_n,\Proba_{\vartheta_n(s)}\right)-\psi\left(s+e_n,\Proba_{\widehat{\vartheta}(s)}\right) \right]dB^H(s)\|^2\\\\
&+9\Ex\| \displaystyle\int_{-\infty}^{t} U(t+e_n,s+e_n)\left[\psi\left(s+e_n,\Proba_{\widehat{\vartheta}(s)}\right)-\widehat{\psi}\left(s,\Proba_{\widehat{\vartheta}(s)}\right) \right]dB^H(s)\|^2\\\\
&+9\Ex\| \displaystyle\int_{-\infty}^{t}\left[ U(t+e_n,s+e_n)\widehat{\psi}\left(s,\Proba_{\widehat{\vartheta}(s)}\right) - V(t,s)\widehat{\psi}\left(s,\Proba_{\widehat{\vartheta}(s)}\right)\right]dB^H(s)\|^2\\\\
:=& P_1(n)+P_2(n)+P_3(n)
\end{array}
\end{equation*}

For $P_1(n)$, we have
\begin{equation}
\begin{array}{rl}
P_1(n)=&9\,\Ex\| \displaystyle\int_{-\infty}^{t } U(t+e_n ,s+e_n)\left[\psi\left(s+e_n,\Proba_{\vartheta_n(s)}\right)-{\psi}\left(s+e_n,\Proba_{\widehat{\vartheta}(s)}\right)\right]dB^H(s)\|^2\\\\


\le& 9H(2H-1)\displaystyle\int_{-\infty}^{t}\int_{-\infty}^{t}\|U(t+e_n,s+e_n)[\psi(s+e_n,\Proba_{\vartheta_n(s)})-\psi(s+e_n,\Proba_{\widehat{\vartheta}(s)})]\|_{\mathbb{L}_2} \\\\
&\qquad\times
\|U(t+e_n,r+e_n)[\psi(r+e_n,\Proba_{\vartheta_n(r)})-\psi(r+e_n,\Proba_{\widehat{\vartheta}(r)})]\|_{\mathbb{L}_2}|r-s|^{2H-2}dr\,ds\\\\
\le& 9H(2H-1)M^2 (\mathbf{K})^2\displaystyle\int_{-\infty}^{t}\int_{-\infty}^{t}e^{-\delta(t-s)}\mathcal{W}(\Proba_{\vartheta_n(s)},\Proba_{\widehat{\vartheta}(s)}) \\\\
&\qquad\qquad\qquad\times
e^{-\delta(t-r)}\mathcal{W}(\Proba_{\vartheta_n(r)},\Proba_{\widehat{\vartheta}(r)})|r-s|^{2H-2}dr\,ds\\\\
\le& 9H(2H-1)M^2 (\mathbf{K})^2\displaystyle\int_{0}^{+\infty}\int_{0}^{+\infty}e^{-\delta\,y}\,\mathcal{W}(\Proba_{\vartheta_n(t-y)},\Proba_{\widehat{\vartheta}(t-y)}) \\\\
&\qquad\qquad\qquad\times
e^{-\delta\,x}\,\mathcal{W}(\Proba_{\vartheta_n(t-x)},\Proba_{\widehat{\vartheta}(t-x)})|x-y|^{2H-2}dx\,dy\\\\
\le& 9H(2H-1)M^2 (\mathbf{K})^2\displaystyle\int_{0}^{+\infty}\bigg[e^{-\delta\,y}\,\mathcal{W}(\Proba_{\vartheta_n(t-y)},\Proba_{\widehat{\vartheta}(t-y)}) \\\\
&\qquad\times
\displaystyle\left(\int_{0}^{+\infty}e^{-\delta\,x}\,\mathcal{W}^2(\Proba_{\vartheta_n(t-x)},\Proba_{\widehat{\vartheta}(t-x)})dx\right)^{1/2}\left(\int_{0}^{+\infty}e^{-\delta\,x}|x-y|^{4H-4}dx\right)^{1/2}\,\bigg]dy\\\\
\le& 9H(2H-1)M^2 (\mathbf{K})^2\bigg[\displaystyle\int_{0}^{+\infty}e^{-\delta\,y}\,\mathcal{W}^2(\Proba_{\vartheta_n(t-y)},\Proba_{\widehat{\vartheta}(t-y)}) dy\bigg]^{1/2}\\\\
&\times
\displaystyle\left(\int_{0}^{+\infty}e^{-\delta\,x}\,\mathcal{W}^2(\Proba_{\vartheta_n(t-x)},\Proba_{\widehat{\vartheta}(t-x)})dx\right)^{1/2}\bigg[\displaystyle\int_{0}^{+\infty}e^{-\delta\,y}\,\int_{0}^{+\infty}e^{-\delta\,x}|x-y|^{4H-4}dx\,dy\bigg]^{1/2}\\\\
\le& 9H(2H-1)M^2 (\mathbf{K})^2\displaystyle\int_{0}^{+\infty}e^{-\delta\,y}\,\Ex\|\vartheta_n(t-y)-\widehat{\vartheta}(t-y)\|^2 dy\\\\
&\qquad\qquad\times
\left(\displaystyle\int_{0}^{+\infty}e^{-\delta\,y}\,\int_{0}^{+\infty}e^{-\delta\,x}|x-y|^{4H-4}dx\,dy\right)^{1/2}\\\\
\le& 9H(2H-1)M^2 (\mathbf{K})^2\displaystyle\int_{0}^{+\infty}e^{-\delta\,y}\,\Ex\|\vartheta_n(t-y)-\widehat{\vartheta}(t-y)\|^2 dy\\\\
&\times
\left(\displaystyle\int_{0}^{+\infty}e^{-\delta\,y}\,\int_{0}^{y}e^{-\delta\,x}|x-y|^{4H-4}dx\,dy+\displaystyle\int_{0}^{+\infty}e^{-\delta\,y}\,\int_{y}^{+\infty}e^{-\delta\,x}|x-y|^{4H-4}dx\,dy\right)^{1/2}\\\\
\end{array}
\end{equation}
\begin{equation}\label{JP1}
\begin{array}{rl}
\le& 9H(2H-1)M^2 (\mathbf{K})^2\displaystyle\int_{-\infty}^{0}e^{-\delta\,(t-s)}\,\Ex\|\vartheta_n(s)-\widehat{\vartheta}(s)\|^2 ds\,
\left(L_1+L_2\right)^{1/2}
\end{array}
\end{equation}
where
\begin{equation}
\begin{array}{rl}

L_1+L_2=&
\displaystyle\int_{0}^{+\infty}e^{-\delta\,y}\,\int_{0}^{y}e^{-\delta\,x}|x-y|^{4H-4}dx\,dy+\displaystyle\int_{0}^{+\infty}e^{-\delta\,y}\,\int_{y}^{+\infty}e^{-\delta\,x}|x-y|^{4H-4}dx\,dy\\\\

=&
\displaystyle\int_{0}^{+\infty}e^{-\delta\,y}\,\int_{0}^{y}|x-y|^{4H-4}dx\,dy+\displaystyle\int_{0}^{+\infty}e^{-\delta\,y}\,\int_{0}^{+\infty}e^{-\delta\,(y+\tau)}\tau^{4H-4}d\tau\,dy\\\\

=&
\displaystyle\int_{0}^{+\infty}e^{-\delta\,y}\,\dfrac{x^{4H-3}}{4H-3}\,dy+\displaystyle\int_{0}^{+\infty}e^{-\delta\,y}\,\int_{0}^{+\infty}e^{-\delta\,(y+\tau)}\tau^{4H-4}d\tau\,dy\\\\
=&
\dfrac{\Gamma(4H-2)}{(4H-2)\delta^{4H-2}}+\dfrac{\Gamma(4H-3)}{2\delta^{4H-2}}\\\\
\end{array}
\end{equation}
For $P_2(n)$, using similar calculation as in \eqref{JP1}, we obtain

\begin{equation}\label{JP2}
\begin{array}{rl}
P_2(n)=&9\,\Ex\| \displaystyle\int_{-\infty}^{t } U(t+e_n ,s+e_n)\left[{\psi}\left(s+e_n,\Proba_{\widehat{\vartheta}(s)}\right)-\widehat{\psi}\left(s,\Proba_{\widehat{\vartheta}(s)}\right)\right]dB^H(s)\|^2\\\\
\le& 9H(2H-1)\displaystyle\int_{-\infty}^{t}\int_{-\infty}^{t}\|U(t+e_n,s+e_n)[{\psi}\left(s+e_n,\Proba_{\widehat{\vartheta}(s)}\right)-\widehat{\psi}\left(s,\Proba_{\widehat{\vartheta}(s)}\right)]\|_{\mathbb{L}_2} \\\\
&\qquad\qquad\times \,
\|U(t+e_n,r+e_n)[{\psi}\left(r+e_n,\Proba_{\widehat{\vartheta}(r)}\right)-\widehat{\psi}\left(r,\Proba_{\widehat{\vartheta}(r)}\right)]\|_{\mathbb{L}_2}|r-s|^{2H-2}dr\,ds\\\\
\le& 9H(2H-1)\,M^2\displaystyle\int_{-\infty}^{t}\int_{-\infty}^{t}e^{-\delta(t-s)}\|{\psi}\left(s+e_n,\Proba_{\widehat{\vartheta}(s)}\right)-\widehat{\psi}\left(s,\Proba_{\widehat{\vartheta}(s)}\right)\|_{\mathbb{L}_2} \\\\
&\qquad\qquad\times\, e^{-\delta(t-r)}
\|{\psi}\left(r+e_n,\Proba_{\widehat{\vartheta}(r)}\right)-\widehat{\psi}\left(r,\Proba_{\widehat{\vartheta}(r)}\right)\|_{\mathbb{L}_2}|r-s|^{2H-2}dr\,ds\\\\
\le& 9H(2H-1)\,M^2\,(L_1+L_2)^{1/2}\Big(\displaystyle\int_{-\infty}^{t}e^{-\delta(t-s)}\|{\psi}\left(s+e_n,\Proba_{\widehat{\vartheta}(s)}\right)-\widehat{\psi}\left(s,\Proba_{\widehat{\vartheta}(s)}\right)\|^2_{\mathbb{L}_2}ds\Big) \\\\

\end{array}
\end{equation}
Since $\psi$ is square-mean almost automorphic,
$\displaystyle\sup_{t\in \mathbb{R}}\|{\psi}\left(t+e_n,\Proba_{\widehat{\vartheta}(t)}\right)\|_{\mathbb{L}_2}<\infty$ and  $\sup_{t\in \mathbb{R}}\|\widehat{\psi}\left(t,\Proba_{\widehat{\vartheta}(t)}\right)\|^2_{\mathbb{L}_2}<\infty$
then by the Lebesgue dominated convergence theorem, it follows that $$\lim\limits_{n\to \infty}P_2(n)=0.$$
For $P_3(n)$, using similar calculation as in the estimation of $P_1(n)$, we have
\begin{equation*}
\begin{array}{rl}
P_3(n)=&9\Ex\| \displaystyle\int_{-\infty}^{t}\left[ U(t+e_n,s+e_n)\widehat{\psi}\left(s,\Proba_{\widehat{\vartheta}(s)}\right) - V(t,s)\widehat{\psi}\left(s,\Proba_{\widehat{\vartheta}(s)}\right)\right]dB^H(s)\|^2\\\\
\le& 9H(2H-1)\displaystyle\int_{-\infty}^{t}\|U(t+e_n,s+e_n)\widehat{\psi}\left(s,\Proba_{\widehat{\vartheta}(s)}\right) - V(t,s)\widehat{\psi}\left(s,\Proba_{\widehat{\vartheta}(s)}\right)\|_{\mathbb{L}_2} \\\\
&\qquad\qquad\times \,
\bigg[\displaystyle\int_{-\infty}^{t}\left\|U(t+e_n,r+e_n)\widehat{\psi}\left(r,\Proba_{\widehat{\vartheta}(r)}\right) - V(t,r)\widehat{\psi}\left(r,\Proba_{\widehat{\vartheta}(r)}\right)\right\|_{\mathbb{L}_2}|r-s|^{2H-2}dr\,\bigg]ds.
\end{array}
\end{equation*}
Since
\begin{equation*}
\begin{array}{rl}
\displaystyle\int_{-\infty}^{t}&\Bigg(\|U(t+e_n,s+e_n)\widehat{\psi}\left(s,\Proba_{\widehat{\vartheta}(s)}\right) - V(t,s)\widehat{\psi}\left(s,\Proba_{\widehat{\vartheta}(s)}\right)\|_{\mathbb{L}_2} \\\\
&\qquad\qquad\times \,
\bigg[\displaystyle\int_{-\infty}^{t}\left\|U(t+e_n,r+e_n)\widehat{\psi}\left(r,\Proba_{\widehat{\vartheta}(r)}\right) - V(t,r)\widehat{\psi}\left(r,\Proba_{\widehat{\vartheta}(r)}\right)\right\|_{\mathbb{L}_2}|r-s|^{2H-2}dr\,\bigg]\Bigg)ds\\\\
\le& 4M^{2}\displaystyle\int_{-\infty}^{t}\Bigg(e^{-\delta(t-s)}\|\widehat{\psi}\left(s,\Proba_{\widehat{\vartheta}(s)}\right)\|_{\mathbb{L}_2}\int_{-\infty}^{t}e^{-\delta(t-s)} e^{-\delta(t-r)}
\|\widehat{\psi}\left(r,\Proba_{\widehat{\vartheta}(r)}\right)\|_{\mathbb{L}_2}|r-s|^{2H-2}dr\,\Bigg)ds.\\\\
\end{array}
\end{equation*}
Using similar calculation as permormed in $P_1(n)$ we get
\begin{equation*}
\begin{array}{rl}
\displaystyle\int_{-\infty}^{t}\Bigg(e^{-\delta(t-s)}\|\widehat{\psi}\left(s,\Proba_{\widehat{\vartheta}(s)}\right)\|_{\mathbb{L}_2}\int_{-\infty}^{t}&e^{-\delta(t-s)} e^{-\delta(t-r)}
\|\widehat{\psi}\left(r,\Proba_{\widehat{\vartheta}(r)}\right)\|_{\mathbb{L}_2}|r-s|^{2H-2}dr\,\Bigg)ds\\
\leq &\left(L_1+L_2\right)^{1/2}\displaystyle\int_{-\infty}^{0}e^{-\delta\,(t-s)}\,\|\widehat{\psi}\left(s,\Proba_{\widehat{\vartheta}(s)}\right)\|^2_{\mathbb{L}_2} ds\,\\
\le& \displaystyle\frac{\left(L_1+L_2\right)^{1/2}}{\delta}\sup_{s\in \mathbb{R}}\|\widehat{\psi}\left(s,\Proba_{\widehat{\vartheta}(s)}\right)\|^2_{\mathbb{L}_2}<\infty
\end{array}
\end{equation*}
Noting \eqref{EVF}, we have
\begin{equation*}
\Ex\|U(t+e_n,s+e_n)\widehat{\psi}\left(s,\Proba_{\widehat{\vartheta}(s)}\right) - V(t,s)\widehat{\psi}\left(r,\Proba_{\widehat{\vartheta}(s)}\right)\|^2_{\mathbb{L}_2} \to 0 \mbox{ as } n\to \infty.
\end{equation*}
Therefore, due to the Lebesgue dominated convergence theorem, we obtain
 \begin{equation}\label{JP3}
\lim\limits_{n\to \infty} P_3(n)=0.
\end{equation}
From \eqref{J2} and \eqref{JP1}, we deduce that
\begin{equation}
\begin{array}{rl}
&\Ex\| \vartheta_n(t)-\widehat{\vartheta}(t)\|^2\\\\
\le& \dfrac{18\mathbf{K}\,M^2}{\delta}\, \displaystyle\int_{-\infty}^{t}\,e^{-\delta(t-s)}\; \Ex\|\vartheta_n(s)-\widehat{\vartheta}(s)\|^2\,ds \\\\
&+ 18\,\widetilde{C_2}\,M^2\, \mathbf{K} \displaystyle\int_{-\infty}^{t}  e^{-2\delta(t-s)}\Ex\left\| \vartheta_n(s)-\widehat{\vartheta}(s)\right\|^2\,ds \\\\
&+9H(2H-1)M^2 (\mathbf{K})^2\displaystyle\int_{-\infty}^{0}e^{-\delta\,(t-s)}\,\Ex\|\vartheta_n(s)-\widehat{\vartheta}(s)\|^2 ds\,
\left(L_1+L_2\right)^{1/2}+X(n)\\\\
\end{array}
\end{equation}
where $X(n)= X_{1}({n})+ X_2(n)+P_2(n)+P_3(n)$ such that $\lim\limits_{n\to\infty}X(n)=0$.
By Lemma \ref{Gronw}, and the fact that $\left(\dfrac{18\mathbf{K}\,M^2}{\delta}\,+18\,\widetilde{C_2}\,M^2\, \mathbf{K}+9H(2H-1)M^2 (\mathbf{K})^2\right)<1$, it follows that
\begin{equation*}
\Ex\| \vartheta_n(t)-\widehat{\vartheta}(t)\|^2\to 0,\mbox{ as }n\to \infty \mbox{ for each }t\in\mathbb{R}
\end{equation*}
Since $\vartheta(t+e_n)$ has the same distribution as $\vartheta_n(t)$, we derive that $\vartheta(t+e_n)\to \widehat{\vartheta}(t)$ in distribution as $n$ goes to $\infty$. Similarly, we have $\widehat{\vartheta}(t-e_n)\to \vartheta(t)$ in distribution as $n\to\infty$ for each $t\in \mathbb{R}$.
 The proof is complete.

\end{proof}

\section{Weighted pseudo automorphic mild solutions in distribution for Equ.\eqref{C1}}
In this section, we prove the existence and uniqueness of weighted pseudo almost automorphic solutions in distribution for Equ.\eqref{C1}.
 Assume that the following hypotheses hold:
 \begin{description}
 \item[$\mathbf{(H_3)}$] The functions $f: \mathbb{R}\times \Hilbert\times \Borel(\Hilbert)\to \Hilbert$, $\theta : \mathbb{R}\times \Hilbert\times \Borel(\Hilbert)\to \mathbb{L}_2^0$ and $\Psi : \mathbb{R}\times \Borel(\Hilbert)\to \mathbb{L}_2$ are square-mean weighted pseudo almost automorphic in $t\in \mathbb{R}$ with respect to $\rho \in \mathcal{M}^{inv}$ with $f=f_1+\widetilde{f}$, $\theta=\theta_1+\widetilde{\theta}$, $\psi=\psi_1+\widetilde{\psi}$ such that
 $f_1:\mathbb{R}\times \Hilbert\times \Borel(\Hilbert)\to\Hilbert$,
 $\theta_1\;:\; \mathbb{R}\times \Hilbert\times \Borel(\Hilbert) \to \mathbb{L}_2^0$,
 $\psi_1\, :\,\mathbb{R}\times \Borel(\Hilbert)\to\mathbb{L}_2$ are square almost automorphic process in $t\in\mathbb{R}$ for each $\vartheta\in \mathcal{L}^2(\Proba, \Hilbert)$ and $\widetilde{f}(\cdot,\vartheta(\cdot),\Proba_{\vartheta(\cdot)})\in SBC_0(\mathbb{R},\rho)$, $\widetilde{\theta}(\cdot,\vartheta(\cdot),\Proba_{\vartheta(\cdot)})\in SBC_0(\mathbb{R},\rho)$  and $\widetilde{\psi}(\cdot,\Proba_{\vartheta(\cdot)})\in SBC_0(\mathbb{R}\times \Borel(\Hilbert,\rho)$ for each $\vartheta\in \mathcal{L}^2(\Proba, \Hilbert)$  .
 \item[$\mathbf{(H_4)}$]For all $\vartheta,\widetilde{\vartheta}\in \mathcal{L}^2(\Proba,\Kilbert)$, $\nu_1,\nu_2\in \Borel(\Hilbert)$ and $t\in \mathbb{R}$, there exists a constant $\mathbf{K}>0$ such that
 $$
 \begin{array}{rl}
\|f(t,\vartheta(t),\nu_1)-f(t,\widetilde{\vartheta}(t),\nu_2)\|^2\le& \mathbf{K} \left(\|\vartheta(t)-\widetilde{\vartheta}(t)\|^2 + \mathcal{W}^2(\nu_1,\nu_2) \right),\\\\
\|f_1(t,\vartheta(t),\nu_1)-f_1(t,\widetilde{\vartheta}(t),\nu_2)\|^2\le& \mathbf{K} \left(\|\vartheta(t)-\widetilde{\vartheta}(t)\|^2 + \mathcal{W}^2(\nu_1,\nu_2) \right),\\\\
 \|\theta(t,\vartheta(t),\nu_1)-\theta(t,\widetilde{\vartheta}(t),\nu_2)\|^2_{\mathbb{L}_2^0}\le& \mathbf{K} \left(\|\vartheta(t)-\widetilde{\vartheta}(t)\|^2 + \mathcal{W}^2(\nu_1,\nu_2) \right),\\\\
  \|\theta_1(t,\vartheta(t),\nu_1)-\theta_1(t,\widetilde{\vartheta}(t),\nu_2)\|^2_{\mathbb{L}_2^0}\le& \mathbf{K} \left(\|\vartheta(t)-\widetilde{\vartheta}(t)\|^2 + \mathcal{W}^2(\nu_1,\nu_2) \right),\\\\
\|\psi(t,\nu_1)-\psi(t,\nu_2)\|_{\mathbb{L}_2}\le& \mathbf{K}\;  \mathcal{W}(\nu_1,\nu_2),\\\\
\|\psi_1(t,\nu_1)-\psi_1(t,\nu_2)\|_{\mathbb{L}_2}\le& \mathbf{K}\;  \mathcal{W}(\nu_1,\nu_2) .
\end{array}
$$
 \end{description}

\begin{theorem}\label{Th2} Suppose that conditions $\mathbf{(H_{0})}$, $\mathbf{(H_3)}$ and $\mathbf{(H_4)}$  hold. Then Equ.\eqref{C1} has a unique $\mathcal{L}^2$-bounded solution provided that
	\begin{equation}\label{Cond21}
	2\mathbf{K}M^2\left(\dfrac{1}{\delta^2}+\dfrac{\widetilde{C_2}}{2\delta}\right)<1
	\end{equation} and
	\begin{equation}\label{Cond21+}
	\dfrac{\beta_2}{\delta}\left[1+\,\dfrac{\beta_2}{\delta}\right]<1,
	\end{equation}
	where $\beta_2$ is a positive constant (see \eqref{beta12}).\\	Furthermore, this unique $\mathcal{L}^2$-bounded solution is weighted pseudo almost automorphic in distribution.%
\end{theorem}
\begin{proof}
	Consider the operator  $\mathbf{\mathcal{S}} : \mathcal{C}_b(\mathbb{R},\mathcal{L}^2(\Proba,\Hilbert)) \to \mathcal{C}_b(\mathbb{R},\mathcal{L}^2(\Proba,\Hilbert))$ defined by
	\begin{equation}
	\begin{array}{ll}
	(\mathbf{\mathcal{S}}\vartheta)(t)=&\displaystyle\int_{-\infty}^{t} U(t,s)f(s,\vartheta(s),\Proba_{\vartheta(s)})ds + \displaystyle\int_{-\infty}^{t} U(t,s)\theta(s,\vartheta(s),\Proba_{\vartheta(s)})dW(s)\\\\
	& + \displaystyle\int_{-\infty}^{t} U(t,s)\psi(s,\Proba_{\vartheta(s)})dB^H(s).
	\end{array}
	\end{equation}
From the \textbf{Step 1} in the proof of Theorem \ref{Th1}, we derive that $\mathbf{\mathcal{S}}$ is a contraction mapping in $\mathcal{C}_b(\mathbb{R},\mathcal{L}^2(\Proba,\Hilbert))$. Thus, $\mathbf{\mathcal{S}}$ has a unique fixe point $\vartheta^{\star}(t)$.


	By condition $\mathbf{(H_3)}$,  there exist
	$f_1\in SAA(\mathbb{R}\times \Hilbert\times \Borel(\Hilbert),\Hilbert)$,
	 $\theta_1\in SAA(\mathbb{R}\times \Hilbert\times \Borel(\Hilbert),\mathbb{L}_2^0)$,
	$\psi_1\in SAA(\mathbb{R}\times \Borel(\Hilbert),\mathbb{L}_2)$, $\widetilde{f}\in SBC_0(\mathbb{R}\times \Hilbert\times \Borel(\Hilbert,\rho)$, $\widetilde{\theta}\in SBC_0(\mathbb{R}\times \Hilbert\times \Borel(\Hilbert,\rho)$  and $\widetilde{\psi}\in SBC_0(\mathbb{R}\times \Borel(\Hilbert,\rho)$  such that
	\begin{equation*}
	\begin{array}{rl}
f(t,\vartheta(t),\mu(t))&=f_1(t,\vartheta(t),\mu(t))+\widetilde{f}(t,\vartheta(t),\mu(t)),\\ \theta(t,\vartheta(t),\mu(t))&=\theta_1(t,\vartheta(t),\mu(t)+\widetilde{\theta}(t,\vartheta(t),\mu(t)),\\
\psi(t,\mu(t))&=\psi_1(t,\vartheta(t),\mu(t)+\widetilde{\psi}(t,\vartheta(t),\mu(t))
	\end{array}
	\end{equation*}
for all $\vartheta\in \mathcal{C}_b(\mathbb{R},\mathcal{L}^2(\Proba,\Hilbert)) $ and $\mu \in \mathcal{C}_b(\mathbb{R},\Borel(\Hilbert))$.
We have
\begin{equation*}
\begin{array}{rl}
(\mathbf{\mathcal{S}}\vartheta^{\star})(t)=&\displaystyle\int_{-\infty}^{t} U(t,s)f(s,\vartheta^{\star}(s),\Proba_{\vartheta^{\star}(s)})ds + \displaystyle\int_{-\infty}^{t} U(t,s)\theta(s,\vartheta^{\star}(s),\Proba_{\vartheta^{\star}(s)})dW(s)\\\\
& + \displaystyle\int_{-\infty}^{t} U(t,s)\psi(s,\Proba_{\vartheta^{\star}(s)})dB^H(s)\\\\
=&\Big[\displaystyle\int_{-\infty}^{t} U(t,s)f_1(s,\vartheta^{\star}(s),\Proba_{\vartheta^{\star}(s)})ds + \displaystyle\int_{-\infty}^{t} U(t,s)\theta_1(s,\vartheta^{\star}(s),\Proba_{\vartheta^{\star}(s)})dW(s)\\\\
& + \displaystyle\int_{-\infty}^{t} U(t,s)\psi_1(s,\Proba_{\vartheta^{\star}(s)})dB^H(s)\Big]
+\Big[\displaystyle\int_{-\infty}^{t} U(t,s)\widetilde{f}(s,\vartheta^{\star}(s),\Proba_{\vartheta^{\star}(s)})ds\\\\
& + \displaystyle\int_{-\infty}^{t} U(t,s)\widetilde{\theta}(s,\vartheta^{\star}(s),\Proba_{\vartheta^{\star}(s)})dW(s) + \displaystyle\int_{-\infty}^{t} U(t,s)\widetilde{\psi}(s,\Proba_{\vartheta^{\star}(s)})dB^H(s)\Big]\\\\
\vartheta^{\star}(t)=&\vartheta^{\star}_{1}(t) +\vartheta^{\star}_{2}(t),
\end{array}
\end{equation*}
where \begin{equation*}
\begin{array}{rl}
\vartheta^{\star}_{1}(t)=&\displaystyle\int_{-\infty}^{t} U(t,s)f_1(s,\vartheta^{\star}(s),\Proba_{\vartheta^{\star}(s)})ds + \displaystyle\int_{-\infty}^{t} U(t,s)\theta_1(s,\vartheta^{\star}(s),\Proba_{\vartheta^{\star}(s)})dW(s)\\\\
&+ \displaystyle\int_{-\infty}^{t} U(t,s)\psi_1(s,\Proba_{\vartheta^{\star}(s)})dB^H(s)
\end{array}
\end{equation*} and
\begin{equation*}
\begin{array}{rl}
\vartheta^{\star}_{2}(t)=&\displaystyle\int_{-\infty}^{t} U(t,s)\widetilde{f}(s,\vartheta^{\star}(s),\Proba_{\vartheta^{\star}(s)})ds
+ \displaystyle\int_{-\infty}^{t} U(t,s)\widetilde{\theta}(s,\vartheta^{\star}(s),\Proba_{\vartheta^{\star}(s)})dW(s)\\\\
& + \displaystyle\int_{-\infty}^{t} U(t,s)\widetilde{\psi}(s,\Proba_{\vartheta^{\star}(s)})dB^H(s).
\end{array}
\end{equation*}
Using $\mathbf{(H_{0})}$, $\mathbf{(H_3)}$ and $\mathbf{(H_4)}$, it follows from Theorem \ref{Th1} that $\vartheta^{\star}_{1}$ is almost automorph in distribution.
In order to show that $\vartheta^{\star}(t)$ is a square-mean weighted pseudo almost automorphic process, it is sufficient to  prove  that $\vartheta^{\star}_{2}\in SBC_0(\mathbb{R}, \rho)$.\\
Similar to the \textbf{Step 3} in the proof of Theorem \ref{Th1}, $\vartheta^\star_2(t)$ is $\mathcal{L}^2$-continuous and $\mathcal{L}^2$-bounded.\\
To conclude the proof, we must check that \begin{equation*}
\lim\limits_{q\to +\infty} \dfrac{1}{m(q,\rho)}\displaystyle\int_{-q}^{q}\Ex\|\vartheta^\star_2(t)\|^2\rho(t)\,dt=0.
\end{equation*}
Observe that
\begin{equation}
\begin{array}{rl}
&  \dfrac{1}{m(q,\rho)}\displaystyle\int_{-q}^{q}\Ex\|\vartheta^\star_2(t)\|^2\rho(t)\,dt\\\\
\le&   \dfrac{3}{m(q,\rho)}\displaystyle\int_{-q}^{q}\Ex\left\|\displaystyle\int_{-\infty}^{t} U(t,s)\widetilde{f}(s,\vartheta^{\star}(s),\Proba_{\vartheta^{\star}(s)})ds\,\right\|^2\,\rho(t)\,dt\\\\
&+  \dfrac{3}{m(q,\rho)}\displaystyle\int_{-q}^{q}\Ex\left\| \displaystyle\int_{-\infty}^{t} U(t,s)\widetilde{\theta}(s,\vartheta^{\star}(s),\Proba_{\vartheta^{\star}(s)})dW(s)\right\|^2\,\rho(t)\,dt\\\\
& +   \dfrac{3}{m(q,\rho)}\displaystyle\int_{-q}^{q}\Ex\left\|\displaystyle\int_{-\infty}^{t} U(t,s)\widetilde{\psi}(s,\Proba_{\vartheta^{\star}(s)})dB^H(s)\right\|^2\,\rho(t)\,dt\\\\
\end{array}
\end{equation}
By an argument  similar the one used in the the proof of \cite[Theorem 4.1]{Chen} with minor modifications, we get
\begin{equation}\label{LIM1}
\begin{array}{rl}
&      \lim\limits_{q\to +\infty} \dfrac{3}{m(q,\rho)}\displaystyle\int_{-q}^{q}\Ex\left\|\displaystyle\int_{-\infty}^{t} U(t,s)\widetilde{f}(s,\vartheta^{\star}(s),\Proba_{\vartheta^{\star}(s)})ds\,\right\|^2\,\rho(t)\,dt\\\\
&+ \lim\limits_{q\to +\infty}  \dfrac{3}{m(q,\rho)}\displaystyle\int_{-q}^{q}\Ex\left\| \displaystyle\int_{-\infty}^{t} U(t,s)\widetilde{\theta}(s,\vartheta^{\star}(s),\Proba_{\vartheta^{\star}(s)})dW(s)\right\|^2\,\rho(t)\,dt\\\\
&=0 .
\end{array}
\end{equation}
On the other hand, an argument similar  to the proof of Theorem 6.1 in Ref \cite{Mbaye} with minor modifications, enables us to conclude that
\begin{equation*}
\begin{array}{rl}
&\Ex\big\|\displaystyle\int_{-\infty}^t U(t,s) \widetilde{\psi}(s,\Proba_{\vartheta^{\star}(s)})\,d B^H(s)\big\|^2\\\\
\leq &C(H,M)\big\|\widetilde{\psi}\big\|_{\infty}\big\{ \displaystyle\int_{0}^{\infty}e^{-\delta v}\big\|\widetilde{\psi}(t-v,\Proba_{\vartheta^{\star}(t-v)})\big\|_{{\mathbb L}_2}\frac{v^{2H-1}}{2H-1}\,dv\\\\
&\quad \quad~~~~~~~~~~~~~~~~~~~~ + \dfrac{ \Gamma(2H-1)}{\delta^{2H-1}}\displaystyle\int_{0}^{\infty}e^{-2\delta v}\big\| \widetilde{\psi}(t-v,\Proba_{\vartheta^{\star}(t-v)})\big\|_{{\mathbb L}_2}\,dv\big\}\\\\
&\leq B_{1}\displaystyle\int_{0}^{\infty}e^{-\delta v}\big\| \widetilde{\psi}(t-v,\Proba_{\vartheta^{\star}(t-v)})\big\|_{{\mathbb L}_2}v^{2H-1}\,dv
+ B_{2}\displaystyle\int_{0}^{\infty}e^{-2\delta v}\big\| \widetilde{\psi}(t-v,\Proba_{\vartheta^{\star}(t-v)})\big\|_{{\mathbb L}_2}\,dv,
\end{array}
\end{equation*}
where
\begin{align*}
&C(H,M)=H(2H-1)M^2,\\ &B_{1}:=\frac{C(H,M)\big\|\widetilde{\psi}\big\|_{\infty}}{2H-1} \\ &B_{2}:=\frac{C(H,M)\big\|\widetilde{\psi}\big\|_{\infty} \Gamma(2H-1)}{\delta^{2H-1}}.
\end{align*}
We have
\begin{equation*}
\begin{array}{rl}
&\dfrac{1}{m(q,\rho)}\displaystyle\int_{-q}^{q}\Ex\left\|\displaystyle\int_{-\infty}^{t} U(t,s)\widetilde{\psi}(s,\Proba_{\vartheta^{\star}(s)})dB^H(s)\right\|^2\,\rho(t)\,dt\\\\
=&\dfrac{1}{m(q,\rho)}\displaystyle\int_{-q}^{q}B_{1}\displaystyle\int_{0}^{\infty}e^{-\delta v}\big\| \widetilde{\psi}(t-v,\Proba_{\vartheta^{\star}(t-v)})\big\|_{{\mathbb L}_2}v^{2H-1}\,dv\,\rho(t)\,dt\\\\
&+\dfrac{1}{m(q,\rho)}\displaystyle\int_{-q}^{q}B_{2}\displaystyle\int_{0}^{\infty}e^{-2\delta v}\big\| \widetilde{\psi}(t-v,\Proba_{\vartheta^{\star}(t-v)})\big\|_{{\mathbb L}_2}\,dv\,\rho(t)\,dt.
\end{array}
\end{equation*}
By the Fubini theorem, we get
\begin{equation*}
\begin{array}{rl}
&\dfrac{1}{m(q,\rho)}\displaystyle\int_{-q}^{q}\Ex\left\|\displaystyle\int_{-\infty}^{t} U(t,s)\widetilde{\psi}(s,\Proba_{\vartheta^{\star}(s)})dB^H(s)\right\|^2\,\rho(t)\,dt\\\\
=&B_{1}\displaystyle\int_{0}^{\infty}e^{-\delta v}v^{2H-1}\,dv\dfrac{1}{m(q,\rho)}\displaystyle\int_{-q}^{q}\big\| \widetilde{\psi}(t-v,\Proba_{\vartheta^{\star}(t-v)})\big\|_{{\mathbb L}_2}\,\rho(t)\,dt\\\\
&+B_{2}\displaystyle\int_{0}^{\infty}e^{-2\delta v}\,dv\dfrac{1}{m(q,\rho)}\displaystyle\int_{-q}^{q}\big\| \widetilde{\psi}(t-v,\Proba_{\vartheta^{\star}(t-v)})\big\|_{{\mathbb L}_2}\,\rho(t)\,dt.
\end{array}
\end{equation*}
Since $\rho \in \mathcal{M}^{inv}$ and $\widetilde{\psi}\in SBC_0(\mathbb{R},\rho)$, we obtain that for any $v\in \mathbb{R}$,
\begin{equation*}
\dfrac{1}{m(q,\rho)}\displaystyle\int_{-q}^{q}\big\| \widetilde{\psi}(t-v,\Proba_{\vartheta^{\star}(t-v)})\big\|_{{\mathbb L}_2}\,\rho(t)\,dt \to 0 \mbox{ as } q\to +\infty.
\end{equation*}

Then, by the Lebesgue dominated convergence theorem, we have

\begin{equation*}
\begin{array}{rl}
&B_{1}\displaystyle\int_{0}^{\infty}e^{-\delta v}v^{2H-1}\,dv\dfrac{1}{m(q,\rho)}\displaystyle\int_{-q}^{q}\big\| \widetilde{\psi}(t-v,\Proba_{\vartheta^{\star}(t-v)})\big\|_{{\mathbb L}_2}\,\rho(t)\,dt \to 0
\end{array}
\end{equation*}
and
\begin{equation*}
\begin{array}{rl}
&B_{2}\displaystyle\int_{0}^{\infty}e^{-2\delta v}\,dv\dfrac{1}{m(q,\rho)}\displaystyle\int_{-q}^{q}\big\| \widetilde{\psi}(t-v,\Proba_{\vartheta^{\star}(t-v)})\big\|_{{\mathbb L}_2}\,\rho(t)\,dt \to 0 \\\\
\end{array}
\end{equation*}
as $q\to +\infty$. Hence,

\begin{equation}\label{LIM2}
\lim\limits_{q\to +\infty}\dfrac{1}{m(q,\rho)}\displaystyle\int_{-q}^{q}\Ex\left\|\displaystyle\int_{-\infty}^{t} U(t,s)\widetilde{\psi}(s,\Proba_{\vartheta^{\star}(s)})dB^H(s)\right\|^2\,\rho(t)\,dt=0
\end{equation}

By  \eqref{LIM1} and \eqref{LIM2}, we deduce that \begin{equation*}
\lim\limits_{q\to +\infty} \dfrac{1}{m(q,\rho)}\displaystyle\int_{-q}^{q}\Ex\|\vartheta^\star_2(t)\|^2\rho(t)\,dt=0.
\end{equation*}
Therefore, $\vartheta^\star_2\in SBC_0(\mathbb{R},\rho)$. The proof is complete.

\end{proof}

\section{Examples}
\subsection*{Example 1}
 Let $(\Omega,\F,\mathbb{P},\mathcal{G}_{t})$ be a filtered probability space. $B^H(t)$  a cylindrical fractional Brownian motion with Hurst parameter $H\in (1/2, 1),$ and $\mathbb{W}(t)$ a two-sided standard one-dimensional Brownian motion independent of $B^H(t)$ on $L^2[0,1]$. Consider the following one-dimensional stochastic heat equation
\begin{equation}\label{Exp}
\left\{\begin{array}{rl}
du(t,x)&= \left[\dfrac{\partial^{2}}{\partial x^{2}} + \sin\left(\dfrac{1}{2+\sin(t)+\sin(\pi\,t)}\right) \right]u(t,x)dt \\\\

&+c_1\left[\dfrac{1}{2}\sin\left( \dfrac{1}{2+\cos(t)+\cos(\sqrt{2}t)}\right)\left( \dfrac{u(t,x)}{u^2(t,x)+1}+\mathcal{W}(\mathbb{P}_0,\Proba_{u(t,x)})\right)+ b(t)\cos(u)\right] dt \\\\
&+c_2 \left[\dfrac{1}{2}\sin\left( \dfrac{1}{2+\cos(t)+\cos(t\sqrt{3})}\right)\left(u(t,x)+\mathcal{W}(\mathbb{P}_0,\Proba_{u(t,x)})\right)+b(t)\sin(u(t,x))\right] d\mathbb{W}(t)\\\\
& +c_3\left[\dfrac{1}{2}\sin\left( \dfrac{1}{2+\cos(t)+\cos(t\sqrt{3})}\right)\mathcal{W}(\mathbb{P}_0,\Proba_{u(t,x)})\right]dB^{H}(t),\\\\
&\hspace{2cm} \mbox{ for all }(t,x) \in \mathbb{R}\times(0,1),\\
\\
u(t,0)&=u(t,1)=0 ~~\mbox{for}\quad t \in \mathbb{R},\\
\end{array}
\right.
\end{equation}

where $b(t)=t.1_{[0,1]}(t)+t.1_{[1,\infty)}(t)$, $1_{J}(\cdot)$ is a characteristic function on the interval $J$ and $c_i \;(i=1,2,3)$ are positive constants, $\mathbb{P}_0$ is a regular probability distribution on $L^{2}(0,1)$, and $\mathcal{W}$ is the Wasserstein distance.
\par In order to write the system \eqref{Exp} in the abstract form \eqref{C1}, we set $\Hilbert:=L^{2}(0,1)$ and consider the linear operator $A: D(A)\subset \Hilbert\rightarrow \Hilbert$, defined by
\begin{align*}
D(A)&=H^{2}(0,1)\cap H^{1}_{0}(0,1),\\
Az(\xi)&=z"(\xi) \quad\mbox{for}\quad \xi \in (0,1)\quad\mbox{and}\quad z \in D(A).
\end{align*}
It is well-known that A generates a $C_0$-semigroup $(R(t))_{t\geq0}$ on $\Hilbert$ that satisfies  $\|R(t)\|\leq e^{-\pi^{2}t}$ for all $t\geq0$.\\
Define a family of linear operator $A(t)$ as follows:
\begin{equation*}
\left\{\begin{array}{ll}
D(A(t))&=D(A),\\\\
A(t)z&=\left[A + \sin\left(\dfrac{1}{2+\sin(t)+\sin(\pi\,t)}\right)\right]z \quad \mbox{for} \quad z \in D(A).
\end{array}
\right.
\end{equation*}
Hence, $\{A(t), t \in \mathbb{R}\}$  generates an evolution family $\{U(t, s), t\ge s\}$ such that
\begin{align*}
U(t, s)z=R(t-s)\exp\Big[\int_{s}^{t}\sin\left(\dfrac{1}{2+\sin(r)+\sin(\pi\,r)}\right) dr \Big]z.
\end{align*}
Since $\|U(t, s)\|\leq e^{-(\pi^{2}- 1)(t-s)}$ for $t\geq s$ and $s,t \in \mathbb{R}$.\\
Choose $M=1$ and $\delta=\pi^2-1$. By the almost automorphic property of $\sin\left(\dfrac{1}{2+\sin(r)+\sin(\pi\,r)}\right)$ and
\begin{equation*}
\begin{array}{rl}
U(t+s_n, s+s_n)z=&R(t-s)\exp\Big[\displaystyle\int_{s+s_n}^{t+s_n}\sin\left(\dfrac{1}{2+\sin(r)+\sin(\pi\,r)}\right) dr \Big]z\\\\
=&R(t-s)\exp\Big[\displaystyle\int_{s}^{t}\sin\left(\dfrac{1}{2+\sin(r+s_n)+
	\sin(\pi\,(r+s_n))}\right)dr \Big]z,\\\\
\end{array}
\end{equation*}
we obtain that $U(t, s)z\in SBAA\left( \mathbb{R}\times \mathbb{R}, \mathcal{L}^2(\Proba,\Hilbert)\right)$ uniformly for all $z$ in any bounded subset of $\mathcal{L}^2(\Proba,\Hilbert)$. Define for all $\ell\in \Hilbert$, $\nu\in \Borel(\Hilbert),x\in(0,1)$ and $t\in \mathbb{R}$

\begin{equation*}
\begin{array}{rl}
f(t,\ell, \nu)(x)=&
\dfrac{c_1}{2}\sin\left( \dfrac{1}{2+\cos(t)+\cos(\sqrt{2}t)}\right)\left( \dfrac{\ell(x)}{\ell^2(x)+1}+\mathcal{W}(\mathbb{P}_0,\nu)\right)+ c_1b(t)\cos(\ell(x))\\\\
\psi(t,\nu)(x)=&\dfrac{c_3}{2}\sin\left( \dfrac{1}{2+\cos(t)+\cos(\sqrt{3}t)}\right)\mathcal{W}(\mathbb{P}_0,\Proba_{\nu(x)})\\ \\
\theta(t,\ell,\nu)(x)=&\dfrac{c_2}{2}\sin\left( \dfrac{1}{2+\cos(t)+\cos(t\sqrt{3})}\right)\left(\ell(x)+\mathcal{W}(\mathbb{P}_0,\Proba_{\nu(x)})\right)+c_2b(t)\sin(\ell(x))\\
\end{array}
\end{equation*}
Setting $\vartheta(t)(x)=u(t,x)
$, the system \eqref{Exp} can be rewritten in the abstract form
\begin{equation*}
d\vartheta(t)=A(t)\vartheta(t)dt + f(t, \vartheta(t),\Proba_{\vartheta(t)})\,dt + \theta(t, \vartheta(t),\Proba_{\vartheta(t)})\,d\mathbb{W}(t) +\psi(t,\Proba_{\vartheta(t)})\,dB^H(t),\quad t \in \mathbb{R}.
\end{equation*}
Choose $\rho(t)=e^{-t}$, then $\rho\in \mathcal{M}^{inv}$. It is easy to show that
$f$, $\psi$ and $\theta$ are square-mean weighted pseudo almost automorphic processes about $\rho(t)=e^{-t}$.
The functions $f$, $\psi$ and $\theta$ satisfies the global Lipschitz condition, with Lipschitz constants $\mathbf{K}=2\max\left\{ \dfrac{c_1^2}{4}+c_1^2,\;\dfrac{c_2^2}{4}+c_2^2,\; \dfrac{c_3}{2} \right\}$. By calculation and appropriate condition on $c_i\,( i=1,2,3)$ (say for $c_1,c_2,c_3$ are small enough) and the Hurst parameter $H$, conditions \eqref{Cond21} and \eqref{Cond21+} of Theorem \ref{Th2} hold and so \eqref{Exp} has a unique weighted pseudo almost automorphic solution in distribution.

\subsection*{Example 2}

Consider a
McKean-Vlasov autonomous stochastic evolution equation of the form

\begin{align}
dr(t)x&=\left(\dfrac{\partial}{\partial x}r(t)(x) +\left<\zeta_1(t, r(t),\Proba_{r(t)})(x),
\displaystyle\int_0^{|x|}\zeta_1(t, r(t),\Proba_{r(t)})(y)\right>_{\mathsf{V}}\right)
dt\notag\\
&\quad + <\zeta_2(t, r(t),\Proba_{r(t)})(x), dW(t)>_{\mathsf{V}},\quad t,x \in\Real,\label{MeanHJMM}
\end{align}
where $(\mathsf{V}, <\cdot, \cdot>_{\mathsf{V}})$ is a Hilbert separable space and $W$ is a two-side
$\mathsf{V}$-valued Wiener process on $(\Omega,\F,\mathbb{P},\mathcal{G}_{t})$ be a filtered probability space.

For some functions, $\zeta_1$ and $\zeta_2$ are specified below, and some certain spaces of functions, we shall prove the existence and the uniqueness of almost automorphic solution to problem \eqref{MeanHJMM}. Notice that a mean field Heath-Jarrow-Morton-Musiela (HJMM) equation fits perfectly in the framework of problem \eqref{MeanHJMM}. For more details, see \cite{Jarrow,Musiela}. We will analyse this equation, for certain functions $\zeta_1$ and $\zeta_2$.\\
For each $\nu>0$, let $L^2_{\nu}$ be the space of all (equivalence classes of) Lebesgue measurable functions
$u : \Real \rightarrow\Real$ such that
$$\int_{\Real}| u (x)|^2 e^{\nu x} dx < \infty.$$
It is well-known that for each $\nu \in\Real$, $L^2_{\nu}$ is a Hilberrt space endowed with the norm
$$ \| u\| _{\nu,2} =\left(
\int_{\Real}| u (x)|^2 e^{\nu x} dx\right)^{1/2}.$$
Set $\mathsf{U}=L^2_{\nu}$ and define the shift group $S$ as in the following lemma.
\begin{lemma}
 Let $S = \{S (t)\}_{t\in\Real}$ be a family of operators on $L^2_{\nu}$ defined by
$$S (t) u (x) = u (t + x),\quad u \in \mathsf{U} ,\quad t,x\in\Real.$$
Then, $S$ is a strongly continuous group on $\mathsf{U}$
such that
$$\|S (t)\|_{\Linear(L^2_{\nu}) } \leq e^{-(\nu t)/2},\quad t\in\Real.$$
Moreover, the infinitesimal generator $A$ of $S$ on
$\mathsf{U} \geq 0$ is given by
$$\mathsf{Dom}(A) =\left\{ u \in L^2_{\nu} : D u \in L^2_{\nu}\right\},\quad A u = D u, u \in \mathsf{Dom}(A),$$
where $D u$ is the first weak derivative of $u$ .
\end{lemma}
\begin{proof}
Following \cite{Ahmad}, it is clear that $S$ is a strongly continuous group on $\mathsf{U}$. For $u\in\mathsf{U}$, we have
\begin{align*}
\|S (t)u\|_{L^2_{\nu}}^2&=\int_{\Real}|(S(t)u) (x)|^2 e^{\nu x} dx\\
&=\int_{\Real}|(u(t+x)|^2 e^{\nu x} dx\\
&=e^{-\nu t}\int_{\Real}|(u(x)|^2 e^{\nu x} dx\\
&= e^{-\nu t}\|u\|_{\mathsf{U}}^2
\end{align*}
From this, the exponential bound follows.
\end{proof}
Assume that function $\zeta$ in equation \eqref{MeanHJMM} is defined by
\begin{align*}
&\zeta_1(t, u, \mu )(x) = \,\int_{\mathsf{U}}\Phi_1(t)g(x, u(x),z)d\mu(z),\quad u \in L^2_{\nu},\quad
t, x \in\Real,\mu\in\Borel(\mathsf{U})\\
&\zeta_2(t, u, \mu )(x)=\Phi_2(t)\,g(x, u(x),0),\quad u \in L^2_{\nu},\quad
t, x \in\Real,\mu\in\Borel(\mathsf{U})
\end{align*}
where $\Phi_1(t)=\cos\left( \dfrac{1}{2+\sin(t)+\sin(\sqrt{3}t)}\right)$, $\Phi_2(t)=\cos\left( \dfrac{1}{2+\sin(t)+\sin(\sqrt{2}t)}\right)$ and $g : \Real^3 \rightarrow \mathsf{V}$ is a given function .
 For $\mu\in\Borel(\mathsf{U}),u \in L^2_{\nu} ,t, x \in \Real$ and $v\in\mathsf{V}$, define
\begin{align}
\label{Gs}
\theta(t, u,\mu )[v](x) = \left<\Phi_2(t)g(x, u(x),0), v\right>_{\mathsf{V}},\quad u \in L^2_{\nu} ,\quad
t, x \in\Real,
\end{align}
and
\begin{align}
\label{Fs}
f(t, u,\mu )(x) =\left< \int_{\mathsf{U}}\Phi_1(t)g(x, u(x),z)d\mu(z),\int_0^{|x|}\int_{\mathsf{U}}\Phi_1(t)g(y, u(y),z)d\mu(z)dy\right>_{\mathsf{V}},
\end{align}

Then, the abstract form of equation \eqref{MeanHJMM} be can written as follows

 \begin{equation}\label{Formulae}
 d\,r(t)= Ar(t)dt+f(t,r(t),\Proba_{r(t)})dt+\theta(t,r(t),\Proba_{r(t)})dW(t), \quad t\in\Real,
 \end{equation}
which is the equation \eqref{C1} with $\psi\equiv0$.
\begin{theorem} Assume that $\nu>0$.  Let $f$ and $\theta$ be as given in \eqref{Fs} and \eqref{Gs}, respectively. Assume that there exist functions
$g_1 \in L^2_{\nu}$ and $g_2 \in L^2_{\nu} \cap L_{\infty}$ such that
$$|g(x,y,z)|_{\mathsf{V}} \leq |g_1(x)|,\quad x,y,z \in\Real,$$
and
$$|g(x, y_1, z) - g(x, y_2, z)|_{\mathsf{V}} \leq |g_2(x)|\left( |y_1 - y_2|+|z_1 - z_2|\right),\quad x,y_1, y_2,z_1,z_2 \in \Real.$$
Then, there exists a unique $L^2_{\nu}$-valued bounded almost automorphic in distribution mild solution $r$ to equation \eqref{Formulae}.

\end{theorem}
\begin{proof}
We have that
\begin{align*}
&f(t, u_1,\mu_1 )(x)-f(t, u_2,\mu_2 )(x)\\
 &=\left< \int_{\mathsf{U}}[\Phi_1(t)g(x, u_1(x),z)-\Phi_1(t)g(x, u_2(x),z)]d\mu_1(z),\int_0^{|x|}\int_{\mathsf{U}}\Phi_1(t)g(y, u_1(y),z)d\mu_2(z)dy\right>_{\mathsf{V}}\\
&\quad+\left< \int_{\mathsf{U}}\Phi_1(t)g(x, u_2(x),z)d\mu_1(z),\int_0^{|x|}\int_{\mathsf{U}}[\Phi_1(t)g(y, u_1(y),z)-\Phi_1(t)g(y, u_2(y),z)]d\mu_2(z)dy\right>_{\mathsf{V}}.
\end{align*}
It follows by the Cauchy-Schwarz and H\"older inequalities that
\begin{align*}
&|f(t, u_1,\mu_1 )(x)-f(t, u_2,\mu_2 )(x)|\\
&\leq|g_2(x)||u_1(x)-u_2(x)|\int_0^{|x|}|g_1(y)|dy+|g_1(x)|\int_0^{|x|}|g_2(y)||u_1(y)-u_2(y)|dy\\
&\leq\|g_2\|_{L_{\infty}}|u_1(x)-u_2(x)|\int_0^{\infty}|g_1(y)|dy+|g_1(x)|\int_0^{\infty}\|g_2\|_{L_{\infty}}|u_1(y)-u_2(y)|dy\\
&\leq\|g_2\|_{L_{\infty}}|u_1(x)-u_2(x)|\int_0^{\infty}|g_1(y)|e^{(\nu y)/2}e^{-(\nu y)/2}dy\\
&\quad+\|g_2\|_{L_{\infty}}|g_1(x)|\int_0^{\infty}|u_1(y)-u_2(y)|e^{(\nu y)/2}e^{-(\nu y)/2}dy\\
&\leq\|g_2\|_{L_{\infty}}\|g_1\|_{L^2_{\nu}}|u_1(x)-u_2(x)|\left(\int_0^{\infty}e^{-\nu y}dy\right)^{1/2}\\
&\quad+\|g_2\|_{L_{\infty}}|g_1(x)|\|u_1-u_2\|_{L^2_{\nu}} \left(\int_0^{\infty}e^{-\nu y}dy\right)^{1/2}\\
&\leq\frac{1}{\sqrt{\nu}}\|g_2\|_{L_{\infty}}\|g_1\|_{L^2_{\nu}}|u_1(x)-u_2(x)|+\frac{1}{\sqrt{\nu}}\|g_2\|_{L_{\infty}}\|u_1-u_2\|_{L^2_{\nu}}|g_1(x)|
\end{align*}
We deduce from this that
\begin{align*}
&\|f(t, u_1,\mu_1 )-f(t, u_2,\mu_2 )\|_{L^2_{\nu}}\leq2\sqrt{\frac{2}{\nu}}\|g_2\|_{L_{\infty}}\|g_1\|_{L^2_{\nu}}\|u_1 -u_2\|_{L^2_{\nu}}
\end{align*}
Thus,
\begin{align*}
\|f(t, u_1,\mu_1 )-f(t, u_2,\mu_2 )\|_{L^2_{\nu}}^2&\leq\frac{8}{\nu}\|g_2\|_{L_{\infty}}^2\|g_1\|_{L^2_{\nu}}^2\|u_1 -u_2\|_{L^2_{\nu}}^2\\
&\leq\frac{8}{\nu}\|g_2\|_{L_{\infty}}^2\|g_1\|_{L^2_{\nu}}^2 \left\{\|u_1 -u_2\|_{L^2_{\nu}}^2+\mathcal{W}^2(\mu_1,\mu_2)\right\}\\
&=:\mathbf{K}_1\left\{\|u_1 -u_2\|_{L^2_{\nu}}^2+\mathcal{W}^2(\mu_1,\mu_2)\right\}.
\end{align*}
where $\mathbf{K}_1=\frac{8}{\nu}\|g_2\|_{L_{\infty}}^2\|g_1\|_{L^2_{\nu}}^2.$\\
Note that  $v\mapsto\theta(t, u,\mu )[v]$ is a bounded linear map for $(t, u,\mu )\in\Real^2\times\Borel(\mathsf{U})$ and
\begin{align*}
&\theta(t, u_1,\mu_1 )[v](x)-\theta(t, u_2,\mu_2 )[v](x)\\
&=\left<\int_{\mathsf{U}}\Phi_2(t)g(x, u_1(x),z)d\mu_1(z), v\right>_{\mathsf{V}}-\left<\int_{\mathsf{U}}\Phi_2(t)g(x, u_2(x),z)d\mu_1(z), v\right>_{\mathsf{V}}\\
&+\left<\int_{\mathsf{U}}\Phi_2(t)g(x, u_2(x),z)d\mu_1(z), v\right>_{\mathsf{V}}-\left<\int_{\mathsf{U}}\Phi_2(t)g(x, u_2(x),z)d\mu_2(z), v\right>_{\mathsf{V}}
\end{align*}
Applying Cauchy-Schwarz inequality yields
\begin{align*}
&|\theta(t, u_1,\mu_1 )[v](x)-\theta(t, u_2,\mu_2 )[v](x)|\\
&\leq\left|\left< (\Phi_2(t)g(x, u_1(x),0)-\Phi_2(t)g(x, u_2(x),0)), v\right>_{\mathsf{V}}\right|\\
&\leq|g_2(x)|| u_1(x)- u_2(x)| \|v\|_{\mathsf{V}}\\
&\leq\|g_2\|_{L_{\infty}} \|v\|_{\mathsf{V}}| u_1(x)- u_2(x)|.
\end{align*}
Then it follows that
\begin{align*}
\|\theta(t, u_1,\mu_1 )-\theta(t, u_2,\mu_2 )\|_{L^2_{\nu}}^2&\leq\|g_2\|_{L_{\infty}}^2\| u_1- u_2\|_{L^2_{\nu}}^2\\
&\leq\|g_2\|_{L_{\infty}}^2 \left\{\|u_1 -u_2\|_{L^2_{\nu}}^2+\mathcal{W}^2(\mu_1,\mu_2)\right\}\\
&=:\mathbf{K}_2\left\{\|u_1 -u_2\|_{L^2_{\nu}}^2+\mathcal{W}^2(\mu_1,\mu_2)\right\}.
\end{align*}

Therefore, we set $\mathbf{K}=\max\left(\mathbf{K}_1,\mathbf{K}_2\right) M=1$ and $ \delta=\frac{\nu}{2}$.
For $\|g_2\|_{L_{\infty}}^2$ and $\|g_1\|_{L^2_{\nu}}^2$ small enough and $\nu$ big enough, the conditions \eqref{Cond1}-\eqref{Cond2} hold. So by Theorem \ref{Th1}, Then, Equ.\eqref{Formulae} has a unique $\mathcal{L}^2$-bounded solution  almost automorphic in distribution.
\end{proof}

\end{document}